\documentclass[a4paper, twoside,10pt]{article}
\usepackage{amsmath}
\usepackage{amssymb}
\usepackage{amsfonts}
\usepackage{graphicx}
\usepackage{subfigure}
\usepackage{eucal}
\usepackage{amsthm}
\usepackage{verbatim}
\usepackage{epsfig}

\def\Rc{\mathbb{R}}
\def\Cc{\mathbb{C}}

\newtheorem{thm}{Theorem}

\newtheorem{cor}{Corollary}

\newtheorem{remark}{Remark}
\newtheorem{conjecture}{Conjecture}

\begin{document}

\title{
Evaluation of small elements of the eigenvectors of certain
symmetric tridiagonal
matrices with high relative accuracy}
\author{Andrei Osipov\footnote{Yale University, 
51 Prospect st, New Haven, CT 06511.
Email: andrei.osipov@yale.edu.
}}
\maketitle

\begin{abstract}
Evaluation of the eigenvectors of symmetric tridiagonal matrices
is one of the most basic tasks in numerical linear algebra.
It is a widely known fact that, in the case of well separated
eigenvalues, the eigenvectors can be evaluated with high
relative accuracy. Nevertheless, in general, each coordinate
of the eigenvector is evaluated with only high 
\emph{absolute} accuracy.
In particular, those coordinates whose magnitude is
below the machine precision
are not expected to be evaluated 
with any accuracy whatsoever.

It turns out that, under certain conditions,
frequently ecountered in applications,
small
(e.g. $10^{-50}$) coordinates of eigenvectors of
symmetric tridiagonal matrices can be evaluated
with high \emph{relative} accuracy.
In this paper, we 
investigate such conditions, carry out the analysis,
and describe the resulting numerical schemes.
While our schemes can be viewed
as a modification of already existing (and well known) 
numerical algorithms,
the related error analysis appears to be new.
Our results are illustrated
via several numerical examples.
\end{abstract}


\noindent
{\bf Math subject classification:} {65G99, 65F15, 65Q30}


\tableofcontents

\section{Introduction}
\label{sec_intro}

The evaluation of eigenvectors of symmetric tridiagonal matrices
is one of the most basic tasks in numerical linear algebra
(see, for example, such classical texts as 
\cite{Dahlquist}, \cite{Francis}, \cite{Givens}, \cite{Golub},
\cite{Isaacson}, \cite{Kublanovskaya},
\cite{Parlett}, \cite{Stoer_Bul}, \cite{Wilkinson}).
Several algorithms to perform this task have been developed;
these include Power and Inverse Power methods, Jacobi Rotations,
QR and QL algorithms, to mention just a few.
Many of these algorithms have become standard and widely known tools.

In the case when the eigenvalues of the matrix in question
are well separated, most of these algorithms will evaluate
the corresponding eigenvectors to a high \emph{relative} accuracy.
More specifically, suppose that $n>0$ is an integer,
that $A$ is an $n$ by $n$ symmetric matrix,
that $\lambda$ is an eigenvalue of $A$, 
that $v \in \Rc^n$ is the corresponding unit-length eigenvector,
and  that
$\hat{v} \in \Rc^n$ is its numerical approximation (produced
by one of the standard algorithms).
Then,
\begin{align}
\| v - \hat{v} \| \leq M \cdot \varepsilon,
\label{eq_mat_intro_rel_acc}
\end{align}
where $\| \cdot \|$ denotes the Euclidean norm,
$\varepsilon$ is the machine precision
(e.g. $\varepsilon \approx 10^{-16}$ for double precision calculations),
and $M$ is proportional to the inverse of the distance between
$\lambda$ and the rest of the spectrum of $A$.

However, a closer look at \eqref{eq_mat_intro_rel_acc} reveals
that it only guarantees that the \emph{coordinates} of $v$ be
evaluated to high \emph{absolute} accuracy. This is due to the following
trivial observation. Suppose that
we add $\varepsilon$ 
to the first coordinate $\hat{v}_1$ of $\hat{v}$.
Then,
the perturbed $\hat{v}$ will not violate \eqref{eq_mat_intro_rel_acc}.
On the other hand, the relative accuracy of $\hat{v}_1$ can be as
large as
\begin{align}
\frac{ |v_1 + \varepsilon - v_1| }{ |v_1| } =
\frac{\varepsilon}{ |v_1| }.
\label{eq_mat_intro_rel_v1}
\end{align}
In particular, if $|v_1| < \varepsilon$,
then $\hat{v}_1$ is not guaranteed to approximate
$v_1$ with any relative accuracy whatsoever.

Sometimes the poor relative accuracy of "small" coordinates
is of no concern; for example, this is usually
the case
when $v$ is only used to project other vectors onto it. 
Nevertheless, in several prominent problems, small coordinates
of the eigenvector often need to be evaluated to high
relative accuracy. 
Numerical evaluation of special functions
provides a rich source of such problems;
these include the evaluation
of Bessel functions (see Sections~\ref{sec_prel_bessel},
\ref{sec_num_bessel}, \ref{sec_mat_bessel}),
the evaluation of some quantities
associated with prolate spheroidal wave functions
(see Section~\ref{sec_mat_pswf_num},
and also \cite{Report5ACHA}), and
the evaluation of singular values of the truncated Laplace transform
(see \cite{Lederman}),
among others.

In this paper, we describe a scheme for the evaluation
of the coordinates of eigenvectors of certain symmetric tridiagonal
matrices, to high relative accuracy. More specifically,
we consider the matrices whose non-zero off-diagonal elements
are constant (or approximately so), and whose diagonal elements
constitute a monotonically increasing sequence
(see, however, Remark~\ref{rem_intro_general} below).
The connection
of such matrices to Bessel functions and prolate spheroidal wave functions
is discussed in Sections~\ref{sec_num_bessel}, \ref{sec_mat_pswf_num},
respectively. Also, we carry out detailed error analysis of our algorithm
(see Sections~\ref{sec_about_errors}, \ref{sec_asym}).
While our scheme can be viewed as a modification
of already existing (and well known) algorithms, such error analysis,
perhaps surprisingly, appears to be new. In addition, we
conduct several numerical experiments to illustrate
the analysis, to demonstrate our scheme's accuracy,
and to compare the latter to that
of some classical algorithms (see Section~\ref{sec_numerical}).

The following is one of the principal analytical results of this paper
(see Theorem~\ref{thm_ac} 
in Section~\ref{sec_asym} for a more precise statement,
and Theorems~\ref{thm_summary_left}, \ref{thm_summary_right}, \ref{thm_glue},
Corollary~\ref{cor_glue} 
in Section~\ref{sec_about_errors} below 
for the treatment of a more general case).
\begin{thm}
Suppose that $a \geq 1$ is a real number, and that,
for any real $c \geq 1$, $n=n(c)>c$ is an integer,
the real numbers $A_1(c),\dots,A_n(c)$ are
defined via the formula
\begin{align}
A_j(c) = 2 + 2 \cdot \left( \frac{j}{c} \right)^a,
\label{eq_intro_principal_aj}
\end{align}
for every $j=1, \dots, n$, and that the $n$ by $n$
symmetric tridiagonal matrix $A=A(c)$ is defined via the formula
\begin{align}
A(c) = 
\begin{pmatrix}
A_1 & 1 &   &   &   &   \\
1 & A_2 & 1 &   &   &   \\ 
  & 1 & A_3 & 1 &   &   \\
  &   & \ddots & \ddots & \ddots & \\
  &   &   & 1 & A_{n-1} & 1 \\
  &   &   &   & 1 & A_n \\
\end{pmatrix}.
\label{eq_intro_principal_matrix}
\end{align}
Suppose furthermore that, for any real $c>1$, 
$\lambda(c)$ is an eigenvalue of $A(c)$, that $1<k(c)<n(c)$ is an integer,
that
\begin{align}
2 + A_{k(c)} < \lambda(c) \leq 2 + A_{k(c)+1}.
\label{eq_intro_principal_lambda}
\end{align}
and that $X(c)=(X_1,\dots,X_n) \in \Rc^n$ is 
the unit-length $\lambda(c)-$eigenvector of $A(c)$. 
Suppose, in addition, that $\varepsilon>0$, and that the entries of $A(c)$
are defined to relative precision $\varepsilon$, for any $c>1$.
Then,
the first $k(c)$ coordinates $X_1, \dots, X_k$ 
of $X(c)$ are defined to the relative
precision $R(c,a)$, where
\begin{align}
R(c,a) \leq \varepsilon \cdot 
O\left( -\log(X_1) \right) \cdot O\left( c^{4\cdot a/(a+2)} \right), \quad
c \to \infty.
\label{eq_intro_principal}
\end{align}
\label{thm_intro_principal}
\end{thm}
\begin{remark}
We observe that, according to \eqref{eq_intro_principal},
the relative precision of $X_1,\dots,X_k$ depends only
logarithmically
on their order of magnitude. In other words, even if, say,
$X_1$ is significantly smaller than $\varepsilon$, it is
still defined to fairly high relative precision.
\label{rem_intro_principal}
\end{remark}

\begin{remark}
The definition of the entries of the matrix $A$ in 
Theorem~\ref{thm_intro_principal} is motivated by 
particular applications
(see Section~\ref{sec_applications}). On the other hand,
Theorem~\ref{thm_intro_principal} and Remark~\ref{rem_intro_principal}
generalize to a much wider class of matrices; these include,
for example, perturbations
of $A$, defined via \eqref{eq_intro_principal_matrix}; 
matrices whose diagonal entries are of a more general
form than \eqref{eq_intro_principal_aj}; banded
(not necessarily tridiagonal) matrices, etc.
While such generalizations are straightforward (and are based,
in part, on the results of Section~\ref{sec_vectors}),
the analysis is somewhat involved, and will be published 
at a later date (see, however, 
Theorems~\ref{thm_summary_left}, \ref{thm_summary_right} and
Corollary~\ref{cor_glue} 
in Section~\ref{sec_about_errors} below for one such generalization).
\label{rem_intro_general}
\end{remark}

The proof of Theorem~\ref{thm_intro_principal} is constructive
and somewhat technical (see Sections~\ref{sec_about_errors}, \ref{sec_asym}).
The resulting numerical
algorithms for the evaluation of the eigenvector $X$
are described in Section~\ref{sec_num_algo}.

In practice, the upper bound in \eqref{eq_intro_principal} above
seems to be overly pessimistic. In fact, the following conjecture
has been verified by extensive numerical experiments
(see Section~\ref{sec_numerical}).
\begin{conjecture}
Suppose that, in addition to 
the hypothesis of Theorem~\ref{thm_intro_principal} above,
we evaluate the eigenvector $X(c)$ by
the algorithm from Section~\ref{sec_short}.
Then, for any real $c>1$ and any integer $1 \leq j \leq k(c)$,
\begin{align}
\text{\rm{rel}}(X_j(c)) \leq 100 \cdot c^{2 \cdot a/(a+2)} \cdot \varepsilon.
\label{eq_exp1_20}
\end{align}
In particular,
$\text{\rm{rel}}(X_1)$ does not depend on the magnitude of $X_1$,
for any $c>1$.
\label{conj_exp1}
\end{conjecture}

We observe that the power of $c$ in \eqref{eq_exp1_20}
is half the power of $c$ in \eqref{eq_intro_principal}.
In other words, Theorem~\ref{thm_intro_principal} appears
to overestimate the number of lost digits in the evaluation
of the first $k$ elements of $X$ by roughly a factor of two.

The paper is organized as follows.
In Section~\ref{sec_prel},
we summarize a number of well known mathematical and numerical facts
to be used in the rest of this paper.
In Section~\ref{sec_analytical}, we develop
the necessary analytical apparatus and perform error analysis
of the algorithm, described in
Section~\ref{sec_num_algo}
(and we also describe a number of related algorithms).
In Section~\ref{sec_applications}, we discuss some
applications of our algorithm to other computational problems.
In Section~\ref{sec_numerical}, we illustrate the numerical stability
of our algorithm and corresponding theoretical results
via several numerical examples, and provide
comparison to some related
classical algorithms.

\section{Mathematical and Numerical Preliminaries}
\label{sec_prel}
In this section, we introduce notation and summarize
several facts to be used in the rest of the paper.

\subsection{Bessel Functions}
\label{sec_prel_bessel}

In this section, we describe some well known properties of Bessel functions.
All of these properties can be found, for example,
in \cite{Abramovitz}, \cite{Ryzhik}.

Suppose that $n \geq 0$ is a non-negative integer. The Bessel function
of the first kind $J_n : \Cc \to \Cc$ is defined via the formula
\begin{align}
J_n(z) = \sum_{m=0}^{\infty} \frac{(-1)^m}{m! \cdot (m+n)!} \cdot
   \left(\frac{z}{2}\right)^{2m+n},
\label{eq_bessel_jn}
\end{align}
for all complex $z$. Also, the function $J_{-n} : \Cc \to \Cc$ is
defined via the formula
\begin{align}
J_{-n}(z) = (-1)^n \cdot J_n(z),
\label{eq_bessel_jneg}
\end{align}
for all complex $z$.

The Bessel functions $J_0, J_{\pm 1}, J_{\pm 2}, \dots$ satisfy the 
three-term recurrence relation
\begin{align}
z \cdot J_{n-1}(z) + z \cdot J_{n+1}(z) = 2n \cdot J_n(z),
\label{eq_bessel_rec}
\end{align}
for any complex $z$ and every integer $n$.
In addition, 
\begin{align}
\sum_{n=-\infty}^{\infty} J^2_n(x) = 1,
\label{eq_bessel_sum}
\end{align}
for all real $x$.

\subsection{Numerical Tools}
\label{sec_num_tools}
In this subsection, we summarize several numerical techniques
to be used in this paper.

\subsubsection{Shifted Inverse Power Method}
\label{sec_power}
\index{Inverse power method, shifted}
Suppose that $n \geq 0$ is an integer, and that $A$ is an $n$ by $n$
real symmetric matrix. Suppose also that
$\sigma_1 < \sigma_2 < \dots < \sigma_n$ are the eigenvalues of $A$.
The Shifted Inverse Power Method iteratively
finds the eigenvalue $\sigma_k$ and the corresponding
eigenvector $v_k \in \Rc^n$, 
provided that an approximation $\lambda$ to $\sigma_k$
is given, and that
\begin{align}
|\lambda - \sigma_k| < \max \left\{
   |\lambda - \sigma_j| \; : \; j \neq k \right\}.
\label{eq_inverse_power_sigma}
\end{align}
Each Shifted Inverse Power iteration solves the linear system
\begin{align}
\left(A - \lambda_j I\right) \cdot x = w_j
\end{align}
in the unknown $x \in \Rc^n$, where $\lambda_j$ and $w_j \in \Rc^n$
are the approximations to $\sigma_k$ and $v_k$, respectively,
after $j$ iterations; 
the number $\lambda_j$ is usually referred to as "shift".
The approximations $\lambda_{j+1}$ and $w_{j+1} \in \Rc^n$ 
(to $\sigma_k$ and $v_k$, respectively) are evaluated
from $x$ via the formulae
\begin{align}
w_{j+1} = \frac{ x }{ \| x \|}, \quad
\lambda_{j+1} = w_{j+1}^T \cdot A \cdot w_{j+1}
\end{align}
(see, for example, 
\cite{Dahlquist}, \cite{Wilkinson} for more details).
\begin{remark}
\label{rem_power}
For symmetric matrices, the Shifted Inverse Power Method converges cubically
in the vicinity of the solution.
In particular, if
the matrix $A$ is tridiagonal, and the initial approximation $\lambda$
is sufficiently close to $\sigma_k$, the Shifted Inverse Power Method
evaluates $\sigma_k$ and $v_k$ essentially to machine precision
$\varepsilon$
in $\displaystyle O\left(\log (-\log \varepsilon) \right)$ 
iterations, and each iteration requires
$O(n)$ operations (see e.g
\cite{Wilkinson}, \cite{Dahlquist}).
\end{remark}

\subsubsection{Evaluation of Bessel Functions}
\label{sec_num_bessel}

Suppose that $x>0$ is a real number, and that $m>x$ is an integer.
The classical scheme for the evaluation of
$J_0(x),J_1(x),\dots,J_m(x)$
is based on \eqref{eq_bessel_jneg}, \eqref{eq_bessel_rec}, \eqref{eq_bessel_sum}
in Section~\ref{sec_prel_bessel}
(see e.g \cite{Abramovitz}, \cite{Miller}) consists of the following steps.

\begin{itemize}
\item select integer $N>\max\left\{m,x\right\}$ 
(see Remark~\ref{rem_bessel_n} below).
\item set $X_N=1$ and $X_{N+1}=0$.
\item evaluate $X_{N-1}, X_{N-2}, \dots, X_1, X_0$
iteratively
via the recurrence relation \eqref{eq_bessel_rec}, in the direction
of decreasing indices. In other words, 
for every $k=N,\dots, 1$, evaluate $X_{k-1}$ via
the formula
\begin{align}
X_{k-1} = \frac{2k}{x} \cdot X_k(x) - X_{k+1}(x),
\label{eq_bessel_rec_num}
\end{align}
\item for every $k=0,\dots,m$, evaluate 
the approximation $\tilde{J}_k$ to $J_k(x)$ via
\begin{align}
\tilde{J}_k = X_k \cdot 
\left(X^2_0 + 2 \cdot \sum_{l=1}^N X^2_l\right)^{-\frac{1}{2}}.
\label{eq_bessel_d}
\end{align}
\end{itemize}
\begin{remark}
In this paper, we always select sufficiently large $N$ so that the algorithm
described above, when carried out in extended precision, evaluates
$J_0(x),\dots,J_m(x)$ to at least 17 decimal digits. Further discussion
of the matter is beyond the scope of this paper (see e.g.
\cite{Abramovitz}, \cite{Miller} for more details).
\label{rem_bessel_n}
\end{remark}
%

\section{Analytical Apparatus}
\label{sec_analytical}
The purpose of this section is to provide
the analytical apparatus to be used in the
rest of the paper.

\subsection{Local Properties of Eigenvectors of Certain Tridiagonal Matrices}
\label{sec_vectors}
In this subsection, we develop several analytical results pertaining to
the eigenvectors
of certain tridiagonal symmetric matrices.

In the following theorem, we 
describe some obvious properties of the 
eigenvectors of certain tridiagonal symmetric matrices.
\begin{thm}
Suppose that $n>1$ is an integer, that
$2<A_1<A_2<\dots$ 
is an increasing sequence of positive real numbers,
and that the symmetric tridiagonal $n$ by $n$ matrix $A$ is defined via
the formula
\begin{align}
A = 
\begin{pmatrix}
A_1 & 1 &   &   &   &   \\
1 & A_2 & 1 &   &   &   \\ 
  & 1 & A_3 & 1 &   &   \\
  &   & \ddots & \ddots & \ddots & \\
  &   &   & 1 & A_{n-1} & 1 \\
  &   &   &   & 1 & A_n \\
\end{pmatrix}.
\label{eq_running}
\end{align}
Suppose also that the
real number $\lambda$ is an eigenvalue of $A$, 
and that $x = (x_1,\dots,x_n) \in \Rc^n$
is an eigenvector corresponding to $\lambda$.
Then,
\begin{align}
x_2 & \; = \left(\lambda-A_1\right) \cdot x_1.
\label{eq_main_rec_first}
\end{align}
Also,
\begin{align}
x_{j+1} = \left(\lambda-A_j\right) \cdot x_j - x_{j-1},
\label{eq_main_rec}
\end{align}
for every $j=2,\dots,n-1$. Finally,
\begin{align}
x_{n-1} & \; = \left(\lambda-A_n\right) \cdot x_n.
\label{eq_main_rec_last}
\end{align}
In particular, both $x_1$ and $x_n$ differ from zero,
and $\lambda$ is simple.
\label{thm_main_rec}
\end{thm}
\begin{proof}
The identities \eqref{eq_main_rec_first},
\eqref{eq_main_rec}, \eqref{eq_main_rec_last} follow immediately from
\eqref{eq_running} and the fact that
\begin{align}
A \cdot x = \lambda \cdot x.
\label{eq_main_rec_1}
\end{align}
We observe that the coordinates
$x_2, \dots, x_n$ are
completely determined by $x_1$ and $\lambda$ via \eqref{eq_main_rec_first},
\eqref{eq_main_rec}, and hence
the eigenvalue 
$\lambda$ is simple. Obviously, neither $x_1$ nor $x_n$ can be equal
to zero, for otherwise $x$ would be the zero vector.
\end{proof}

In the following theorem, we assert that,
under certain conditions, the first element
of the eigenvectors of 
the matrix $A$ from Theorem~\ref{thm_main_rec} must be "small".
\begin{thm}
Suppose that the $n$ by $n$ symmetric tridiagonal matrix $A$
is defined via \eqref{eq_running} in Section~\ref{sec_vectors}.
Suppose also that $\lambda$ is an eigenvalue of $A$, and that 
$x=(x_1,\dots,x_n)\in\Rc^n$ is a corresponding eigenvector
whose first coordinate is positive, i.e. $x_1>0$. Suppose, in addition,
that $1 \leq k \leq n$ is an integer, and that
\begin{align}
\lambda \geq A_k + 2.
\label{eq_vec_grow_lam}
\end{align}
Then,
\begin{align}
0 < x_1 < x_2 < \dots < x_k < x_{k+1}.
\label{eq_vec_grow_xs}
\end{align}
Also,
\begin{align}
\frac{x_j}{x_{j-1}} > 
\frac{\lambda-A_j}{2} + 
  \sqrt{\left(\frac{\lambda-A_j}{2}\right)^2-1},
\label{eq_vec_grow_adj}
\end{align}
for every $j=2, \dots, k$.
In addition,
\begin{align}
1 < \frac{ x_k}{ x_{k-1} } < \dots < \frac{x_3}{x_2} < \frac{x_2}{x_1}.
\label{eq_vec_grow_rat}
\end{align}
\label{thm_vec_grow}
\end{thm}
\begin{proof}
It follows from \eqref{eq_vec_grow_lam}
that
\begin{align}
\lambda_k - A_1 > \lambda_k - A_2 > \dots > \lambda_k - A_k \geq 2.
\label{eq_vec_grow_1}
\end{align}
We combine \eqref{eq_main_rec_first}, \eqref{eq_main_rec}
in Theorem~\ref{thm_main_rec} with \eqref{eq_vec_grow_1} to obtain
\eqref{eq_vec_grow_xs} by induction. 
Suppose now
that the real numbers $r_1, \dots, r_k$ are defined via the formula
\begin{align}
r_j = \frac{x_{j+1}}{x_j},
\label{eq_vec_grow_2}
\end{align}
for every $j=1, \dots, k$, and that the real numbers
$\sigma_1, \dots, \sigma_k$ are defined via the formula
\begin{align}
\sigma_j = \frac{\lambda-A_j}{2} + 
  \sqrt{\left(\frac{\lambda-A_j}{2}\right)^2-1},
\label{eq_vec_grow_3}
\end{align}
for every $j=1, \dots, k$. In other words, $\sigma_j$ is the largest
root of the quadratic equation
\begin{align}
x^2 - \left(\lambda-A_j\right) \cdot x + 1 = 0.
\label{eq_vec_grow_4}
\end{align}
We observe that 
\begin{align}
\sigma_1 > \dots > \sigma_k \geq 1,
\label{eq_vec_grow_5}
\end{align}
due to
\eqref{eq_vec_grow_1} and \eqref{eq_vec_grow_3}. Also,
\begin{align}
r_1 > \sigma_1 > \sigma_2 > 1,
\label{eq_vec_grow_6}
\end{align}
due to the combination of \eqref{eq_vec_grow_3} and
\eqref{eq_main_rec_first}. Suppose now, by induction, that
\begin{align}
r_{j-1} > \sigma_j > 1.
\label{eq_vec_grow_7}
\end{align}
for some $2 \leq j \leq k-1$. We observe that the roots of
the quadratic equation \eqref{eq_vec_grow_4} are
$1/\sigma_j < 1 < \sigma_j$, 
and combine this observation with \eqref{eq_vec_grow_7}
to obtain
\begin{align}
r_{j-1}^2 - \left(\lambda-A_j\right)\cdot r_{j-1}+1 > 0.
\label{eq_vec_grow_8}
\end{align}
We combine \eqref{eq_vec_grow_8} with \eqref{eq_vec_grow_2}
and \eqref{eq_main_rec} to obtain
\begin{align}
r_j = \frac{x_{j+1}}{x_j} =
\frac{\left(\lambda-A_j\right)\cdot x_j-x_{j-1}}{x_j} =
\lambda-A_j - \frac{1}{r_{j-1}} < r_{j-1}.
\label{eq_vec_grow_9}
\end{align}
Also, we combine 
\eqref{eq_vec_grow_3}, \eqref{eq_vec_grow_7}, \eqref{eq_vec_grow_9} to
obtain
\begin{align}
r_j = \lambda-A_j - \frac{1}{r_{j-1}} >
      \lambda-A_j - \frac{1}{\sigma_j} =
      \frac{\left(\lambda-A_j\right) \cdot \sigma_j-1}{\sigma_j}
     =\sigma_j > \sigma_{j+1}.
\label{eq_vec_grow_10}
\end{align}
In other words, \eqref{eq_vec_grow_7} implies \eqref{eq_vec_grow_10},
and we combine this observation with \eqref{eq_vec_grow_6} to obtain
\begin{align}
r_1 > \sigma_2, \quad r_2 > \sigma_3, \quad \dots, \quad r_{k-1} > \sigma_k.
\label{eq_vec_grow_11}
\end{align}
Also, due to \eqref{eq_vec_grow_9},
\begin{align}
r_1 > r_2 > \dots > r_{k-1}.
\label{eq_vec_grow_12}
\end{align}
We combine \eqref{eq_vec_grow_2}, \eqref{eq_vec_grow_3},
\eqref{eq_vec_grow_11}, \eqref{eq_vec_grow_12}
to obtain \eqref{eq_vec_grow_adj}, \eqref{eq_vec_grow_rat}.
\end{proof}
\begin{cor}
Under the assumptions of Theorem~\ref{thm_vec_grow},
\begin{align}
\frac{x_k}{x_1} > \prod_{j=2}^k \left(
\frac{\lambda-A_j}{2} + 
  \sqrt{\left(\frac{\lambda-A_j}{2}\right)^2-1}
\right).
\label{eq_vec_grow_cor}
\end{align}
\label{cor_vec_grow}
\end{cor}
\begin{remark}
In \cite{Report2ACHA}, the derivation of
an upper bound on the first coordinate
of an eigenvector of a certain matrix is based on 
a generalization of Theorem~\ref{thm_vec_grow}.
\label{rem_vec_grow}
\end{remark}
%
%
In the following theorem, we study the behavior of several last elements
of an eigenvector of the matrix $A$ from Theorem~\ref{thm_main_rec}
above.
\begin{thm}
Suppose that the $n$ by $n$ symmetric tridiagonal matrix $A$
is defined via \eqref{eq_running} in Section~\ref{sec_vectors}.
Suppose also that $\lambda$ is an eigenvalue of $A$, and that 
$x=(x_1,\dots,x_n)\in\Rc^n$ is a corresponding eigenvector
whose last coordinate is positive, i.e. $x_n>0$. Suppose, in addition,
that $1 \leq k \leq n$ is an integer, and that
\begin{align}
\lambda \leq A_k - 2.
\label{eq_vec_decay_lam}
\end{align}
Then,
\begin{align}
0 < |x_n| < |x_{n-1}| < \dots < |x_k| < |x_{k-1}|.
\label{eq_vec_decay_xs}
\end{align}
Also,
\begin{align}
-\frac{x_j}{x_{j+1}} > 
\frac{A_j-\lambda}{2} + 
  \sqrt{\left(\frac{\lambda-A_j}{2}\right)^2-1},
\label{eq_vec_decay_adj}
\end{align}
for every $j=k, \dots, n-1$.
In addition,
\begin{align}
-1 > \frac{ x_k}{ x_{k+1} } > \dots > \frac{x_{n-2}}{x_{n-1}} >
     \frac{x_{n-1}}{x_n}.
\label{eq_vec_decay_rat}
\end{align}
\label{thm_vec_decay}
\end{thm}
\begin{proof}
The proof is essentially identical to that of Theorem~\ref{thm_vec_grow}
above
and will be omitted.
\end{proof}
%
In the rest of this subsection, we investigate the behavior of the
"middle" elements of an eigenvector of the matrix
$A$ from Theorems~\ref{thm_main_rec}, \ref{thm_vec_grow},
\ref{thm_vec_decay} above.
We start with the following theorem.
\begin{thm}
Suppose that $k,m>0$ are integers, that $x_k,\dots,x_{k+m+2}$
are real numbers, that $B_{k+1},\dots,B_{k+m+1}$ are real numbers,
that
\begin{align}
2 > B_{k+1} > \dots > B_{k+m+1} \geq 0,
\label{eq_restate_1}
\end{align}
and that
\begin{align}
x_{j+1} = B_j \cdot x_j - x_{j-1},
\label{eq_restate_2}
\end{align}
for every $j=k+1,\dots,k+m+1$. Suppose also that, for any real number
$0 < \theta \leq \pi/2$, the real $2 \times 2$ matrix $A(\theta)$
is defined via the formula
\begin{align}
A(\theta) = 
\begin{pmatrix}
0 & 1 \\
-1 & 2 \cdot \cos(\theta)
\end{pmatrix}.
\label{eq_mat_8}
\end{align}
Then,
\begin{align}
\begin{pmatrix} x_{j+1} \\ x_{j+2} \end{pmatrix} =
A\left( \arccos\left( \frac{B_{j+1}}{2} \right) \right)
\cdot
\begin{pmatrix} x_j \\ x_{j+1} \end{pmatrix},
\label{eq_restate_3}
\end{align}
for every $j=k,\dots,k+m$.
\label{thm_restate}
\end{thm}
\begin{proof}
The identity \eqref{eq_restate_3} follows from the combination
of \eqref{eq_restate_2} and \eqref{eq_mat_8}.
\end{proof}
%
%
%
%
\begin{thm}
Suppose that $k>0$ and $l>0$ are integers, and that
\begin{align}
0 < \theta_k < \theta_{k+1} < \dots < \theta_{k+l-1} \leq 
\frac{\pi}{4} \cdot \frac{1}{l+3/2}
\label{eq_large_angle_theta0}
\end{align}
are real numbers. Suppose also that $\varepsilon > 0$,
and that the sequence $x_k, \dots, x_{k+l+2}$ is defined via the formulae
\begin{align}
x_k = 1, \quad x_{k+1} = 1+\varepsilon,
\label{eq_large_angle_xk0}
\end{align}
and
\begin{align}
\begin{pmatrix} x_{j+1} \\ x_{j+2} \end{pmatrix} = 
\begin{pmatrix} 0 & 1 \\ -1 & 2\cos(\theta_j) \end{pmatrix} 
\begin{pmatrix} x_j \\ x_{j+1} \end{pmatrix},
\label{eq_large_angle_x_rec0}
\end{align}
for every $j=k,\dots,k+l-1$. 
Then,
\begin{align}
x_k, x_{k+1}, \dots, x_{k+l}, x_{k+l+1} > 0.
\label{eq_large_angle0}
\end{align}
In addition,
\begin{align}
\frac{m+1}{m} >
\frac{x_{k+m+1}}{x_{k+m}} > 
\frac{ \cos\left( (m+1/2)\cdot \theta_{k+l-1} \right) }
     { \cos\left( (m-1/2)\cdot \theta_{k+l-1} \right) },
\label{eq_large_angle_ratio0}
\end{align}
for every integer $m=1,2,\dots,l$; in particular,
\begin{align}
1 + \frac{1}{l} > \frac{x_{k+l+1}}{x_{k+l}} > 1 - \frac{1}{l+3/2}.
\label{eq_large_angle_ratio20}
\end{align}
\label{thm_large_angle0}
\end{thm}
\begin{proof}
We observe that
\begin{align}
\frac{x_{k+m+1}}{x_{k+m}} = 2 \cdot \cos(\theta_{k+m-1}) -
\frac{x_{k+m-1}}{x_{k+m}},
\label{eq_la_0}
\end{align}
for every $m=1,\dots,l$.
We use \eqref{eq_la_0} to 
prove \eqref{eq_large_angle_ratio0} by induction on $m$. For $m=1$,
\begin{align}
\frac{x_{k+2}}{x_{k+1}} = 2 \cdot \cos(\theta_k) - \frac{1}{1+\varepsilon}
< 2,
\label{eq_la_1}
\end{align}
and also
\begin{align}
\frac{\cos(3 \cdot \theta_{k+l-1} / 2)}{\cos(\theta_{k+l-1}/2)} & \; =
4 \cdot \cos(\theta_{k+l-1}/2) - 3 = 
2 \cdot \cos(\theta_{k+l-1}) - 1 \nonumber \\
& \; < 2 \cdot \cos(\theta_k) - 1 < \frac{x_{k+2}}{x_{k+1}}.
\label{eq_la_2}
\end{align}
By induction, for $2 \leq m \leq l$,
\begin{align}
\frac{x_{k+m+1}}{x_{k+m}} < 2 \cdot \cos(\theta_{k+m-1}) - \frac{m-1}{m}
< 2 - \frac{m-1}{m} = \frac{m+1}{m},
\label{eq_la_3}
\end{align}
which proves the left-hand side of \eqref{eq_large_angle_ratio0},
and also
\begin{align}
\frac{x_{k+m+1}}{x_{k+m}} > 2 \cdot \cos(\theta_{k+m-1}) -
\frac{\cos(\theta_{k+l-1} \cdot (m-3/2))}{\cos(\theta_{k+l-1} \cdot (m-1/2))}.
\label{eq_la_4}
\end{align}
However, for any real $\theta$,
\begin{align}
\frac{\cos(\theta \cdot (m-3/2))}{\cos(\theta \cdot (m-1/2))} +
\frac{\cos(\theta \cdot (m+1/2))}{\cos(\theta \cdot (m-1/2))} = 2 \cdot
\cos(\theta),
\label{eq_la_5}
\end{align}
and we combine \eqref{eq_la_4}, \eqref{eq_la_5} to conclude
the right-hand side of \eqref{eq_large_angle_ratio0}.
The inequality \eqref{eq_large_angle_ratio0} implies \eqref{eq_large_angle0}.
Next,
we observe that
\begin{align}
\cos(x)-\sin(x) \geq 1 - \frac{4x}{\pi},
\label{eq_large_angle_8}
\end{align}
for all real $0 \leq x \leq \pi/4$, and combine \eqref{eq_large_angle_8}
with \eqref{eq_large_angle_theta0} to 
obtain
\begin{align}
\frac{ \cos\left( (l+1/2)\cdot \theta_{k+l-1} \right) }
     { \cos\left( (l-1/2)\cdot \theta_{k+l-1} \right) } 
& \; =
\cos\left( \theta_{k+l-1} \right) -
\sin\left( \theta_{k+l-1} \right) \cdot
\tan\left( \theta_{k+l-1} \cdot (l-1/2) \right) 
\nonumber \\
& \; >
\cos\left( \theta_{k+l-1} \right) -
\sin\left( \theta_{k+l-1} \right) \nonumber \\
& \; > 1 - \frac{4}{\pi} \cdot \frac{\pi}{4} \cdot \frac{1}{l+3/2}.
\label{eq_large_angle_9}
\end{align}
Finally, we combine \eqref{eq_large_angle_9} with
\eqref{eq_large_angle_ratio0} to obtain
\eqref{eq_large_angle_ratio20}.
\end{proof}
\begin{cor}
If, in addition to \eqref{eq_large_angle_theta0},
\begin{align}
\left(m + \frac{3}{2} \right) \cdot
\theta_{k+m-1} < \frac{\pi}{4}
\label{eq_cor_large_angle_1}
\end{align}
for every $m=1,\dots,l$, then
\begin{align}
1 + \frac{1}{m} > \frac{x_{k+m+1}}{x_{k+m}} > 1 - \frac{1}{m + 3/2},
\label{eq_cor_large_angle_2}
\end{align}
for every $m=1,\dots,l$.
\label{cor_large_angle_m}
\end{cor}
\begin{remark}
One can prove (along the lines of Theorem~\ref{thm_large_angle0})
that $x_{j+1}>x_j$ for every $j=k,\dots,k+l$, provided that $l<k$ and that
$\varepsilon>k^{-1}$.
\label{rem_kl}
\end{remark}

\begin{thm}
Suppose that $m>0$ is an integer, and $\theta_1,\dots,\theta_m$
are real numbers such that
\begin{align}
0 < \theta_1 < \dots < \theta_m \leq \frac{\pi}{2}.
\label{eq_mat_11}
\end{align}
Suppose also that, for any real number $0 < \theta \leq \pi/2$,
the real $2 \times 2$ matrix $A(\theta)$ is defined via 
\eqref{eq_mat_8}, and
the complex $2 \times 2$ matrices $D(\theta), \Lambda(\theta)$
are defined, respectively, via the formulae
\begin{align}
D(\theta) = \begin{pmatrix}
e^{i \theta} & 0 \\
0 & e^{-i \theta}
\end{pmatrix},
\label{eq_mat_7}
\end{align}
\begin{align}
\Lambda(\theta) = \begin{pmatrix} 
-2 \cdot i \cdot \sin(\theta/2) & 0 \\
0 & 2 \cos(\theta/2)
\end{pmatrix}.
\label{eq_mat_2}
\end{align}
Suppose furthermore that, for any real numbers
$0 < \eta_1, \eta_2 \leq \pi/2$,
the complex $2 \times 2$ matrix $D(\eta_1,\eta_2)$
is defined via the formula
\begin{align}
D(\eta_1,\eta_2) =
\begin{pmatrix}
\sin(\eta_1/2) / \sin(\eta_2/2) & 0 \\
0 & \cos(\eta_1/2) / \cos(\eta_2/2) 
\end{pmatrix},
\label{eq_mat_5}
\end{align}
and that the unitary complex $2 \times 2$ matrix $V$
is defined via the formula
\begin{align}
V = \frac{1}{\sqrt{2}} \cdot
\begin{pmatrix} 
-1 & 1 \\
1 & 1
\end{pmatrix}.
\label{eq_mat_00}
\end{align}
Then,
\begin{align}
& A(\theta_m) \cdot \dots \cdot A(\theta_1) = \nonumber \\
& V \cdot \Lambda(\theta_m) \cdot V \cdot \nonumber \\
& \quad
  D(\theta_m) \cdot V \cdot D(\theta_{m-1},\theta_m) \cdot V \cdot \nonumber \\
& \quad D(\theta_{m-1}) \cdot V \cdot D(\theta_{m-2},\theta_{m-1})
  \cdot V \cdot \nonumber \\
& \quad \dots \nonumber \\
& \quad D(\theta_2) \cdot V \cdot D(\theta_1,\theta_2) 
\cdot V \cdot \nonumber \\
& \quad
D(\theta_1) \cdot V \cdot \Lambda^{-1}(\theta_1) \cdot V.
\label{eq_mat_12}
\end{align}
\label{thm_mat}
\end{thm}
\begin{proof}
Suppose that, for any real number $0 < \theta \leq \pi/2$,
the complex $2 \times 2$ matrix $U(\theta)$ is defined via
the formula
\begin{align}
U(\theta) = \begin{pmatrix} 
1 & e^{i \theta} \\
e^{i \theta} & 1 \\
\end{pmatrix}.
\label{eq_mat_1}
\end{align}
Obviously, $U(\theta)$ admits the decomposition
\begin{align}
U(\theta) = 
e^{i \cdot \theta/2} \cdot V \cdot \Lambda(\theta) \cdot V.
\label{eq_mat_3}
\end{align}
Due to \eqref{eq_mat_3},
the inverse of $U(\theta)$ admits the decomposition
\begin{align}
U(\theta)^{-1} =
e^{-i \cdot \theta/2} \cdot V \cdot \Lambda^{-1}(\theta) \cdot V.
\label{eq_mat_4}
\end{align}
Due to the combination of \eqref{eq_mat_3}, \eqref{eq_mat_4},
\begin{align}
U(\theta_2)^{-1} \cdot U(\theta_1) & \; =
e^{i(\theta_1-\theta_2)/2} \cdot V \cdot \Lambda^{-1}(\theta_2)
\cdot \Lambda(\theta_1) \cdot V \nonumber \\
& \; = e^{i(\theta_1-\theta_2)/2} \cdot V \cdot 
D(\theta_1,\theta_2) \cdot V.
\label{eq_mat_6}
\end{align}
We observe that, for any $0<\theta<\pi$,
\begin{align}
\frac{i}{2 \sin(\theta)} \cdot
\begin{pmatrix} e^{-i \theta} & -1 \\ -1 & e^{-i \theta} \end{pmatrix}
\begin{pmatrix} 0 & 1 \\ -1 & 2\cos(\theta) \end{pmatrix}
\begin{pmatrix} 1 & e^{i \theta} \\ e^{i \theta} & 1 \end{pmatrix} =
\begin{pmatrix} e^{i \theta} & 0 \\ 0 & e^{-i \theta} \end{pmatrix},
\label{eq_decompose_eig}
\end{align}
and combine
\eqref{eq_decompose_eig} with \eqref{eq_mat_3},
\eqref{eq_mat_7}, \eqref{eq_mat_8} to conclude that
\begin{align}
A(\theta) = U(\theta) \cdot D(\theta) \cdot U^{-1}(\theta).
\label{eq_mat_9}
\end{align}
Subsequently, due to the combination of
\eqref{eq_mat_3}, \eqref{eq_mat_4}, \eqref{eq_mat_6}, \eqref{eq_mat_9},
\begin{align}
A(\theta_2) \cdot A(\theta_1) 
& \; =
U(\theta_2) \cdot D(\theta_2) \cdot U^{-1}(\theta_2) \cdot
U(\theta_1) \cdot D(\theta_1) \cdot U^{-1}(\theta_1)
\nonumber \\
& \; =
V \cdot \Lambda(\theta_2) \cdot V \cdot D(\theta_2) \cdot
V \cdot D(\theta_1,\theta_2) \cdot V \cdot D(\theta_1) \cdot V
\cdot \Lambda^{-1}(\theta_1) \cdot V.
\label{eq_mat_10}
\end{align}
Now \eqref{eq_mat_12} follows from \eqref{eq_mat_10}.
\end{proof}
\begin{cor}
Suppose that, for any complex square matrix $A$, we denote
by $\sigma_{\min}(A)$ and $\sigma_{\max}(A)$, respectively, the minimal
and maximal singular values of $A$.
Then, under the assumptions of Theorem~\ref{thm_mat} above,
\begin{align}
& \sigma_{\min}( A(\theta_m) \cdot \dots \cdot A(\theta_1) 
\cdot V \cdot \Lambda(\theta_1)) \geq
2 \cdot \sin \left( \frac{\theta_1}{2} \right), \\
& \sigma_{\max}( A(\theta_m) \cdot \dots \cdot A(\theta_1) 
\cdot V \cdot \Lambda(\theta_1)) \leq
2 \cdot \cos \left( \frac{\theta_1}{2} \right),
\label{eq_mat_12a}
\end{align}
and also
\begin{align}
& \sigma_{\min}( A(\theta_m) \cdot \dots \cdot A(\theta_1) ) \geq
\tan \left( \frac{\theta_1}{2} \right), \\
& \sigma_{\max}( A(\theta_m) \cdot \dots \cdot A(\theta_1) ) \leq
\cot \left( \frac{\theta_1}{2} \right).
\label{eq_mat_12b}
\end{align}
\label{cor_mat}
\end{cor}
\begin{thm}
Suppose, in addition to the hypothesis of Theorem~\ref{thm_mat},
that $\delta>0$ is a real number, and that the vector $x \in \Rc^2$
is defined
via the formula
\begin{align}
x = \begin{pmatrix} 1 \\ 1+\delta \end{pmatrix}.
\label{eq_mat_13}
\end{align}
Then,
\begin{align}
\frac{
\min\left\{ \left| A(\theta_j) \cdot \dots \cdot A(\theta_1) \cdot x \right|
\; : \; 1 \leq j \leq m \right\}
}
{
\max\left\{ \left| A(\theta_j) \cdot \dots \cdot A(\theta_1) \cdot x \right|
\; : \; 1 \leq j \leq m \right\}
}
\geq
\frac{\theta_1}{2},
\label{eq_mat_15c}
\end{align}
and also,
\begin{align}
\frac{ \left| A(\theta_m) \cdot \dots \cdot A(\theta_1) \cdot x \right| }
{|x|} & \; \leq 1 + \frac{1}{2} \cdot 
\frac{(4/\theta_1^2-1) \cdot \delta^2}{(2+\delta)^2+\delta^2 }.
\label{eq_mat_15e}
\end{align}
\label{thm_delta}
\end{thm}
\begin{proof}
Due to the combination of \eqref{eq_mat_2}, \eqref{eq_mat_00} and 
\eqref{eq_mat_13},
\begin{align}
\Lambda^{-1}(\theta_1) \cdot V \cdot x =
\frac{1}{2\sqrt{2}} \cdot
\begin{pmatrix}
i \cdot \delta / \sin(\theta_1/2) \\
(2+\delta) / \cos(\theta_1/2)
\end{pmatrix}.
\label{eq_mat_14}
\end{align}
We combine \eqref{eq_mat_14} with \eqref{eq_mat_12} and
\eqref{eq_mat_12a} to conclude that
\begin{align}
\left| A(\theta_m) \cdot \dots \cdot A(\theta_1) \cdot x \right| \leq 
\frac{1}{\sqrt{2}} \left|
\begin{pmatrix} \delta \cdot \cot(\theta_1/2) \\
2+\delta
\end{pmatrix}
\right|
\leq
\frac{1}{\sqrt{2}} \left|
\begin{pmatrix} 2 \cdot \delta / \theta_1 \\
2+\delta
\end{pmatrix}
\right|.
\label{eq_mat_14a}
\end{align}
and
\begin{align}
\left| A(\theta_m) \cdot \dots \cdot A(\theta_1) \cdot x \right| \geq 
\frac{1}{\sqrt{2}} \left|
\begin{pmatrix} \delta \\
(2+\delta) \cdot \tan(\theta_1/2)
\end{pmatrix}
\right|
\geq 
\frac{1}{\sqrt{2}} \left|
\begin{pmatrix} \delta \\
(2+\delta) \cdot \theta_1/2
\end{pmatrix}
\right|.
\label{eq_mat_14b}
\end{align}
It follows from \eqref{eq_mat_14a} that
\begin{align}
\left| A(\theta_m) \cdot \dots \cdot A(\theta_1) \cdot x \right|^2 \leq
\frac{1}{2} \cdot \left(
(2+\delta)^2 + 
\left( \frac{2 \cdot \delta}{\theta_1} \right)^2
\right).
\label{eq_mat_15a}
\end{align}
Also, it follows from \eqref{eq_mat_14b} that
\begin{align}
\left| A(\theta_m) \cdot \dots \cdot A(\theta_1) \cdot x \right|^2 \geq
\frac{1}{2} \cdot \left(
(2+\delta)^2 + 
\left( \frac{2 \cdot \delta}{\theta_1} \right)^2
\right)
\cdot \frac{ \theta_1^2 }{4}.
\label{eq_mat_15b}
\end{align}
Now \eqref{eq_mat_15c} follows
from the combination of \eqref{eq_mat_15a} and
\eqref{eq_mat_15b}.
Next we observe that, due to \eqref{eq_mat_13},
\begin{align}
|x|^2 = (1+\delta)^2 + 1 = \frac{1}{2} \cdot \left( (2+\delta)^2 + \delta^2
\right).
\label{eq_mat_15d}
\end{align}
We combine \eqref{eq_mat_15a} with \eqref{eq_mat_15d} to conclude that
\begin{align}
\frac{ \left| A(\theta_m) \cdot \dots \cdot A(\theta_1) \cdot x \right| }
{|x|} \leq \sqrt{1 + 
\frac{(4/\theta_1^2-1) \cdot \delta^2}{(2+\delta)^2+\delta^2 }},
\label{eq_mat_15ee}
\end{align}
which implies \eqref{eq_mat_15e}.
\end{proof}
\begin{cor}
Suppose, in addition to the hypotheses of Theorem~\ref{thm_delta},
that $l \geq 1$ is an integer, that
\begin{align}
\theta_1 \cdot \left( l + \frac{5}{2} \right) \geq \frac{\pi}{4}
\label{eq_mat_15f}
\end{align}
(compare to \eqref{eq_large_angle_theta0}), and that
\begin{align}
-\frac{1}{l+3/2} < \delta < \frac{1}{l}
\label{eq_mat_15g}
\end{align}
(see \eqref{eq_large_angle_ratio20}). Then,
\begin{align}
\frac{|x|}{9 \cdot l} <
\left| A(\theta_m) \cdot \dots \cdot A(\theta_1) \cdot x \right| <
4 \cdot |x|.
\label{eq_mat_15ga}
\end{align}
\label{cor_delta}
\end{cor}
\begin{proof}
The right inequality in \eqref{eq_mat_15ga} follows
from the combination of  \eqref{eq_mat_15e}, 
\eqref{eq_mat_15f}, \eqref{eq_mat_15g}.
The left inequality in \eqref{eq_mat_15ga} follows from
the combination of \eqref{eq_mat_15f} and
\eqref{eq_mat_15c}.
\end{proof}

\subsection{Error Analysis}
\label{sec_about_errors}
In Section~\ref{sec_vectors} above, we investigated 
various analytical properties of eigenvectors of
certain tridiagonal symmetric matrices. 
This section deals with stability issues pertaining to
the numerical evaluation of such eigenvectors.

The following theorem is closely related to Theorem~\ref{thm_vec_grow}
in Section~\ref{sec_vectors}.
\begin{thm}
Suppose that $k>2$ is an integer, and that
\begin{align}
B_1 > B_2 > \dots > B_k \geq 2
\label{eq_err_grow_b}
\end{align}
are real numbers. Suppose also that $x_1, \dots, x_{k+1}$
are real numbers defined via the recurrence relation
\begin{align}
& x_1 = 1, \nonumber \\
& x_2 = B_1, \nonumber \\
& x_{j+1} = B_j \cdot x_j - x_{j-1},
\label{eq_err_grow_rec}
\end{align}
for $j\geq 2$,
and that the real numbers $r_1, \dots, r_k$ are defined via the formula
\begin{align}
r_j = \frac{x_{j+1}}{x_j},
\label{eq_err_grow_rs}
\end{align}
for every $j=1,\dots,k$. Then,
\begin{align}
r_j = B_j - \frac{1}{r_{j-1}},
\label{eq_err_grow_r_rec}
\end{align}
for every $j=2,\dots,k$.
\label{thm_err_grow}
\end{thm}
\begin{proof}
The recurrence relation
\eqref{eq_err_grow_r_rec} follows from the combination of 
\eqref{eq_err_grow_rec}, \eqref{eq_err_grow_rs}.
\end{proof}
\begin{thm}
Suppose that $k>2$ is an integer, and that the real numbers
$B_1,\dots,B_k$, $x_1,\dots,x_{k+1}$, $r_1,\dots,r_k$
are those of Theorem~\ref{thm_err_grow} above. Suppose
also that $\varepsilon>0$ is the machine precision,
that $B_1, \dots, B_k$ are defined to machine precision, 
and that $x_1,\dots,x_{k+1}$, $r_1,\dots,r_k$
are calculated, respectively, via \eqref{eq_err_grow_rec},
\eqref{eq_err_grow_r_rec}. Then,
\begin{align}
\text{\rm{rel}}(r_j) \leq (2 \cdot j - 1) \cdot \varepsilon,
\label{eq_err_grow_rel_r}
\end{align}
for every $j=1,\dots,k$,
\begin{align}
\text{\rm{rel}}(x_{j+1}) \leq \varepsilon \cdot j^{2},
\label{eq_err_grow_rel_x}
\end{align}
for every $j=1,\dots,k$, and also
\begin{align}
\text{\rm{rel}}(x_1^2 + x_2^2 + \dots + x_k^2 + x_{k+1}^2) \leq
\varepsilon \cdot 2 \cdot k^{2}.
\label{eq_err_grow_rel_sum}
\end{align}
\label{thm_err_grow_rel}
\end{thm}
\begin{proof}
First, suppose that $\varepsilon_1,\dots, \varepsilon_k$
and $\delta_1,\dots,\delta_k$
are real numbers, that
\begin{align}
|\delta_{j-1}| \leq \varepsilon,
\label{eq_err_grow_rel_0}
\end{align}
for every $j=2,\dots,k$, that
\begin{align}
\hat{r}_{j-1} = r_{j-1} \cdot (1 + \varepsilon_{j-1}),
\label{eq_err_grow_rel_1}
\end{align}
for every $j=2,\dots,k$,
that
\begin{align}
\hat{B}_{j-1} = B_{j-1} \cdot (1 + \delta_{j-1}),
\label{eq_err_grow_rel_1a}
\end{align}
for every $j=2,\dots,k$,
and that
\begin{align}
\hat{r}_{j} = \hat{B}_j - \frac{1}{\hat{r}_{j-1}},
\label{eq_err_grow_rel_2}
\end{align}
for every $j=2,\dots,k$.
Then, due to the combination of \eqref{eq_err_grow_rel_1},
\eqref{eq_err_grow_rel_2}, \eqref{eq_err_grow_r_rec},
\begin{align}
\hat{r}_{j} & \; = \hat{B}_j-\frac{1}{r_{j-1}} + \frac{1}{r_{j-1}} - 
\frac{1}{\hat{r}_{j-1}} \nonumber \\
& \; =
r_j \cdot \left(1 + \frac{\varepsilon_{j-1}}{r_{j-1} \cdot r_j \cdot
(1 + \varepsilon_{j-1})} + \frac{B_j \cdot \delta_j}{r_j} \right).
\label{eq_err_grow_rel_3}
\end{align}
Also, due to Theorem~\ref{thm_vec_grow} in Section~\ref{sec_vectors},
\begin{align}
B_1 = r_1 > r_2 > \dots > r_k > 1,
\label{eq_err_grow_rel_4}
\end{align}
and, moreover, for every $j=1,\dots,k$,
\begin{align}
\frac{B_j}{r_j}<2.
\label{eq_err_grow_rel_5}
\end{align}
We combine \eqref{eq_err_grow_rel_0}, 
\eqref{eq_err_grow_rel_3}, \eqref{eq_err_grow_rel_4},
\eqref{eq_err_grow_rel_5} to conclude 
\eqref{eq_err_grow_rel_r}.
Next, due to \eqref{eq_err_grow_rs},
\begin{align}
x_{j+1} = r_1 \cdot r_2 \cdot r_3 \cdot \dots \cdot r_j,
\label{eq_err_grow_rel_6}
\end{align}
and we combine \eqref{eq_err_grow_rel_6} with \eqref{eq_err_grow_rel_r}
to obtain
\begin{align}
\text{\rm{rel}}(x_{j+1}) \leq \varepsilon \cdot 
\left(1 + 3 + \dots + 2j-1\right),
\label{eq_err_grow_rel_7}
\end{align}
for every $j=1,\dots,k-1$,
which implies \eqref{eq_err_grow_rel_x}.
Finally, due to \eqref{eq_err_grow_rel_x},
\begin{align}
\text{\rm{rel}}(x_1^2 + \dots + x_{k+1}^2) & \leq
\frac{ \sum_{j=1}^k x_{j+1}^2 \cdot (1+\varepsilon \cdot j^{2})^2-
(x_1^2+\dots+x_{k+1}^2)}{x_1^2+\dots+x_{k+1}^2} \nonumber \\
& = \varepsilon \cdot 
\frac{ 
\sum_{j=1}^k x_{j+1}^2 \cdot (2 \cdot j^{2} +
\varepsilon \cdot j^4)
}{x_1^2+\dots+x_{k+1}^2},
\label{eq_err_grow_rel_8}
\end{align}
which implies \eqref{eq_err_grow_rel_sum}.
\end{proof}
%
%

\begin{thm}
Suppose that $k>0$ and $l>0$ are integers, that
\begin{align}
0 < \theta_k < \theta_{k+1} < \dots < \theta_{k+l-1} < 
\frac{\pi}{4} \cdot \frac{1}{l+3/2}
\label{eq_out_theta}
\end{align}
are real numbers, and that the real numbers $B_{k+1},\dots,B_{k+l}$
are defined via the formula
\begin{align}
B_{j+1} = 2 \cdot \cos(\theta_j),
\label{eq_out_b}
\end{align}
for every $j=k,\dots,k+l-1$.
Suppose also that $\varepsilon > 0$,
that the real numbers $x_k, x_{k+1}$ are those of
Theorem~\ref{thm_err_grow} above,
that the sequence $x_{k+2}, \dots, x_{k+l+1}$ is defined via the formula
\begin{align}
x_{j+2} = B_{j+1} \cdot x_{j+1} - x_j,
\label{eq_out_x_rec}
\end{align}
for every $j=k,\dots,k+l-1$,
and that the real numbers $r_k, \dots, r_{k+l}$ are defined
via \eqref{eq_err_grow_rs} for every $j=k,\dots,k+l$.
Suppose furthermore that $\varepsilon>0$ is the machine precision,
that $B_{k+1}, \dots, B_{k+l}$ are defined to precision $\varepsilon$,
and that the precision of $r_k, x_k, x_{k+1}$ is 
described in \eqref{eq_err_grow_rel_r},
\eqref{eq_err_grow_rel_x} 
of Theorem~\ref{thm_err_grow_rel} above.
Then,
\begin{align}
\text{\rm{rel}}(x_{k+m+1}) <
\varepsilon \cdot (k+2\cdot m)^2,
\label{eq_out_17}
\end{align}
for every $m=1,\dots,l$.
Also,
\begin{align}
\text{\rm{rel}}(x_{k+2}^2 + \dots + x_{k+l}^2 + x_{k+l+1}^2) <
2 \cdot \varepsilon \cdot (k+2\cdot l)^2.
\label{eq_out_18}
\end{align}
In addition,
\begin{align}
\text{\rm{rel}}(r_{k+l}) < 4 \cdot (k+l) \cdot \varepsilon.
\label{eq_out_19}
\end{align}
\label{thm_out}
\end{thm}
\begin{proof}
Suppose that the real numbers $C_1,\dots,C_l$ are defined via the formula
\begin{align}
C_j = \frac{ \cos\left( (j-1/2) \cdot \theta_{k+l-1} \right) }
     { \cos\left( (j+1/2) \cdot \theta_{k+l-1} \right) },
\label{eq_out_0}
\end{align}
for every $j=1,\dots,l$.
Then, due to \eqref{eq_large_angle_ratio0},
\begin{align}
\frac{1}{r_{k+j}} = \frac{x_{k+j}}{x_{k+j+1}} < C_j,
\label{eq_out_1}
\end{align}
for every $j=1,\dots,l$.
It follows from \eqref{eq_out_1} that
\begin{align}
\frac{1}{r_{k+1} \cdot \dots \cdot r_{k+m-1}} <
\frac{1}{\cos\left( (m-1/2) \cdot \theta_{k+l-1} \right)},
\label{eq_out_2}
\end{align}
for every $m=2,\dots,l$. Therefore,
\begin{align}
\frac{1}{r_k \cdot r_{k+1}^2 \cdot \dots \cdot r_{k+m-1}^2 \cdot
r_{k+m}} & \; <
\frac{C_m}{\cos^2\left( (m-1/2)\cdot \theta_{k+l-1} \right)} 
\nonumber \\
& \; < \frac{1}{ \cos^2\left( (m+1/2)\cdot \theta_{k+l-1} \right) },
\label{eq_out_3}
\end{align}
for every $m=2,\dots,l$.
We observe that, similar to \eqref{eq_err_grow_r_rec},
\begin{align}
\frac{B_{k+m}}{r_{k+m}} = 1 + \frac{1}{r_{k+m} \cdot r_{k+m-1}},
\label{eq_out_4a}
\end{align}
for every $m = 1, \dots, l$.
Suppose that for every
$j=k,k+1,\dots,k+l$ 
the relative errors of $r_j, B_j$ are denoted,
respectively, by $\varepsilon_j, \delta_j$ 
(similar to \eqref{eq_err_grow_rel_1}, \eqref{eq_err_grow_rel_1a}).
Due to the combination of 
\eqref{eq_vec_grow_xs}, \eqref{eq_err_grow_r_rec},
\eqref{eq_out_4a},
\begin{align}
& \hat{r}_{k+m} = 
\nonumber \\
& r_{k+m} \cdot \left(
1 +
\frac{\varepsilon_{k+m-1}}{r_{k+m-1}\cdot r_{k+m} \cdot (1+\varepsilon_{k+m-1})}
+ \; \delta_{k+m} \cdot \left(
1 + \frac{1}{r_{k+m-1}\cdot r_{k+m}}
\right)
\right),
\label{eq_out_4b}
\end{align}
for every $m=1,\dots,l$. In particular, using \eqref{eq_out_1},
\begin{align}
\varepsilon_{k+1} \leq \varepsilon_k \cdot C_1 + 
\varepsilon \cdot (1+C_1) \leq 
(\varepsilon_k + 2 \varepsilon) \cdot C_1,
\label{eq_out_4}
\end{align} 
and, more generally,
\begin{align}
\varepsilon_{k+m} < (\varepsilon_k + 2 \cdot m \cdot \varepsilon) \cdot
C_1^2 \cdot C_2^2 \cdot \dots C_{k+m-1}^2 \cdot C_{k+m},
\label{eq_out_5}
\end{align}
for every $m=1,\dots,l$.
Next, we combine \eqref{eq_out_5} with \eqref{eq_out_3}
and Theorem~\ref{thm_vec_grow} in Section~\ref{sec_vectors} to conclude that
\begin{align}
\varepsilon_{k+m} < ( \varepsilon_k + 2 \cdot m \cdot \varepsilon) \cdot
{\cos^{-2}\left( (m+1/2) \cdot \theta_{k+l-1} \right) },
\label{eq_out_6}
\end{align}
for every $m=1,\dots,l$. We substitute
\eqref{eq_out_theta} into \eqref{eq_out_6} to obtain the inequality
\begin{align}
\varepsilon_{k+m} < ( \varepsilon_k + 2 \cdot m \cdot \varepsilon) \cdot
\cos^{-2}\left( 
\frac{\pi}{4} \cdot \frac{2m+1}{2l+3}
\right) < 2 \cdot ( \varepsilon_k + 2 \cdot m \cdot \varepsilon),
\label{eq_out_7}
\end{align}
for every $m=1,\dots,l$. In particular, for $m=l$,
\begin{align}
\varepsilon_{k+l} < 2 \cdot 
( \varepsilon_k + 2 \cdot l \cdot \varepsilon).
\label{eq_out_8}
\end{align}
It follows from \eqref{eq_out_7} that
\begin{align}
& \varepsilon_{k+1} + \cdots + \varepsilon_{k+m} < 
2 \cdot m \cdot \varepsilon_k + 2 \cdot m \cdot (m+1) \cdot \varepsilon,
\label{eq_out_12}
\end{align}
for every integer $m=1,\dots,l$. We observe that
\begin{align}
x_{k+m+1} = x_{k+1} \cdot r_{k+1} \cdot \dots \cdot r_{k+m},
\label{eq_out_13}
\end{align}
for every $m>1$, and hence (ignoring the $O(\varepsilon^2)$ terms)
\begin{align}
\text{\rm{rel}}(x_{k+m+1}) <
\text{\rm{rel}}(x_{k+1})  +
2 \cdot m \cdot \varepsilon_k + 2 \cdot m \cdot (m+1) \cdot \varepsilon,
\label{eq_out_14}
\end{align}
for every $m=1,2,\dots,l$. 
We combine \eqref{eq_out_14} with Theorem~\ref{thm_err_grow_rel}
above to obtain \eqref{eq_out_17}, \eqref{eq_out_18},
and combine Theorem~\ref{thm_err_grow_rel} with \eqref{eq_out_8}
to obtain \eqref{eq_out_19}.
\end{proof}

\begin{thm}
Suppose that $k>0$ and $0<l<m$ are integers,
that $\theta_{k+l}, \dots, \theta_{k+m}$ are real numbers such that
\begin{align}
0 < \frac{\pi}{2 \cdot (2 \cdot l + 5)} \leq \theta_{k+l} < \dots
  < \theta_{k+m} \leq \frac{\pi}{2},
\label{eq_osc_1}
\end{align}
that $x_{k+l}, \dots, x_{k+m+2}$
are real numbers, that $x_{k+l}, x_{k+l+1}$ satisfy
\eqref{eq_large_angle_ratio20},
and that $v_{k+l}, \dots, v_{k+m+1}$ are
vectors in $\Rc^2$ defined via the formula
\begin{align}
v_j = \begin{pmatrix} x_j \\ x_{j+1} \end{pmatrix}
\label{eq_osc_2}
\end{align}
for every $j=k+l, \dots, k+m+1$.
Suppose also 
that the real $2 \times 2$ matrices $A(\theta_{k+l}), \dots, A(\theta_{k+m+1})$
are defined via \eqref{eq_mat_8},
and that 
\begin{align}
v_{j+1} = A(\theta_j) \cdot v_j
\label{eq_osc_3}
\end{align}
for every $j=k+l, \dots, k+m$.
Suppose, in addition, that $\varepsilon>0$ is the machine precision,
that $\cos(\theta_{k+j})$ are defined
to relative precision $\varepsilon$ for every $j=l, \dots, m$,
and that $v_{k+l+1}, \dots, v_{k+m}$ are evaluated recursively
via \eqref{eq_osc_3}. 
Then,
\begin{align}
\text{\rm{rel}}(v_j) \leq
9 \cdot l \cdot 
\text{\rm{rel}}(v_{k+l})
 \cdot \frac{\| v_{k+l} \|}{\| v_j \|},
\label{eq_osc_14}
\end{align}
for every $j=k+l+1,\dots,k+m+1$. Also,
\begin{align}
\text{\rm{rel}}(v_j) \leq
81 \cdot l^2 \cdot \text{\rm{rel}}(v_{k+l}),
\label{eq_osc_14a}
\end{align}
for every $j=k+l+1,\dots,k+m+1$.
Finally,
\begin{align}
\text{\rm{rel}}\left(
x_{k+l}^2 + 2 \cdot \left(x_{k+l+1}^2 + \dots + x_{k+m+1}^2\right) +
x_{k+m+2}^2
\right) \leq
162 \cdot l^2 \cdot \text{\rm{rel}}(v_{k+l}).
\label{eq_osc_19}
\end{align}
\label{thm_osc}
\end{thm}
\begin{proof}
Due to the combination of \eqref{eq_osc_1},
\eqref{eq_osc_3} with \eqref{eq_mat_12b}
and \eqref{eq_large_angle_ratio20},
\begin{align}
\text{\rm{rel}}(v_j) \cdot \| v_j \| & \;
\leq \cot\left( \frac{\theta_{k+l}}{2} \right) \cdot
\| v_{k+l} \| \cdot \text{\rm{rel}}(v_{k+l}) 
\leq \frac{2}{\theta_{k+l}} \cdot \| v_{k+l} \| \cdot 
\text{\rm{rel}}(v_{k+l})  \nonumber \\
& \; \leq 9 \cdot l \cdot \| v_{k+l} \| \cdot 
\text{\rm{rel}}(v_{k+l}) ,
\label{eq_osc_7}
\end{align}
for every $j=k+l+1,\dots,k+m+1$, which implies \eqref{eq_osc_14}.
The combination of \eqref{eq_osc_14} and \eqref{eq_mat_15ga}
implies \eqref{eq_osc_14a}.

Thus, ignoring the $O(\varepsilon^2)$ terms,
\begin{align}
\text{\rm{rel}}(\|v_j\|^2) = \text{\rm{rel}}(v_j \cdot v_j)
\leq
2 \cdot \text{\rm{rel}}(v_j) \leq
18 \cdot l \cdot \text{\rm{rel}}(v_{k+l}) \cdot 
\frac{\|v_{k+l}\|}{\|v_j\|},
\label{eq_osc_15}
\end{align}
for every $j=k+l+1,\dots,k+m+1$. Therefore,
\begin{align}
& \text{\rm{rel}}(\|v_{k+l}\|^2 + \dots + \|v_{k+m+1}\|^2) 
\leq \nonumber \\
& 18 \cdot l \cdot \text{\rm{rel}}(v_{k+l}) 
\cdot \| v_{k+l} \|
\cdot \frac{ \|v_{k+l}\| + \dots + \| v_{k+m+1} \| }
   { \|v_{k+l}\|^2 + \dots + \| v_{k+m+1} \|^2 }.
\label{eq_osc_16}
\end{align}
We substitute \eqref{eq_mat_15ga} into \eqref{eq_osc_16}
to obtain 
\begin{align}
\text{\rm{rel}}(\|v_{k+l}\|^2 + \dots + \|v_{k+m+1}\|^2) 
\leq 
18 \cdot 9 \cdot l^2 \cdot \text{\rm{rel}}(v_{k+l}) 
\cdot \| v_{k+l} \|,
\label{eq_osc_16a}
\end{align}
and substitute \eqref{eq_osc_2} into \eqref{eq_osc_16a}
to obtain
\eqref{eq_osc_19}.
\end{proof}
\begin{cor}
Suppose, in addition to the hypothesis of Theorem~\ref{thm_osc},
that the relative accuracy of $x_{k+l}$ satisfies
\eqref{eq_out_17} in Theorem~\ref{thm_out}. Then,
\begin{align}
\text{\rm{rel}}\left(
x_{k+l}^2 + \dots + x_{k+m+2}^2
\right) \leq
162 \cdot l^2 \cdot (k+2\cdot l)^2 \cdot \varepsilon.
\label{eq_osc_20}
\end{align}
\label{cor_osc}
\end{cor}
\begin{proof}
We observe that
\begin{align}
\text{\rm{rel}}( x_{k+m+1}^2 + 2 \cdot x_{k+m+2}^2 ) \leq
2 \cdot \text{\rm{rel}} \left( (x_{k+m+1}, x_{k+m+2})^T \right),
\label{eq_cor_osc_1}
\end{align}
and combine this observation with \eqref{eq_out_17}, \eqref{eq_osc_19}
to obtain \eqref{eq_osc_20}.
\end{proof}

In the following two theorems, we summarize
Theorems~\ref{thm_err_grow}, \ref{thm_err_grow_rel}, \ref{thm_out}, \ref{thm_osc}
and Corollary~\ref{cor_osc} above.
\begin{thm}
Suppose that $k>0$, $l>0$ and $r > k+l$ are integers, that
$B_1,\dots,B_r$ is a sequence of real numbers, that
\begin{align}
B_1 > B_2 > \dots > B_k \geq 2 > B_{k+1} > \dots > B_{k+l} >
2 \cdot \cos\left( \frac{\pi}{4} \cdot \frac{1}{l+3/2} \right)
\label{eq_sum_left_1}
\end{align}
and that
\begin{align}
2 \cdot \cos\left( \frac{\pi}{4} \cdot \frac{1}{l+5/2} \right) \geq
B_{k+l+1} > \dots > B_r \geq 0.
\label{eq_sum_left_2}
\end{align}
Suppose also that $\varepsilon>0$ is the machine precision,
that $B_1,\dots,B_r$ are defined to precision $\varepsilon$,
and that the real numbers
$x_1,x_2,\dots,x_{r+1}$ are evaluated from $B_1,\dots,B_r$
via the recurrence relation \eqref{eq_err_grow_rec}. Then,
\begin{align}
\text{\rm{rel}}(x_j) \leq (j-1)^2 \cdot \varepsilon,
\label{eq_sum_left_3}
\end{align}
for every $j=1,\dots,k+1$. Also, 
\begin{align}
\text{\rm{rel}}(x_{k+1+j}) \leq (k+2\cdot j)^2 \cdot \varepsilon,
\label{eq_sum_left_4}
\end{align}
for every $j=1,\dots,l$. In addition,
\begin{align}
& \text{\rm{rel}}
\label{eq_sum_left_5}
\begin{pmatrix} x_{j-1} \\ x_j \end{pmatrix}
\leq 81 \cdot l^2 \cdot (k+2 \cdot l)^2 \cdot \varepsilon, \\
\label{eq_sum_left_5a}
& \text{\rm{rel}}(x_j) \leq 18 \cdot l \cdot (k+2 \cdot l)^2 
\cdot \left| \frac{ x_{k+l} }{x_j} \right| \cdot \varepsilon
\end{align}
for every $j=k+l+1, \dots, r+1$. Finally,
\begin{align}
\text{\rm{rel}}\left(
x_1^2 + \dots + x_r^2 + x_{r+1}^2
\right) \leq 
162 \cdot l^2 \cdot (k+2\cdot l)^2 \cdot \varepsilon.
\label{eq_sum_left_6}
\end{align}
\label{thm_summary_left}
\end{thm}
\begin{proof}
The combination of
Theorems~\ref{thm_err_grow_rel}, \ref{thm_out}, \ref{thm_osc}
and Corollary~\ref{cor_osc} above.
\end{proof}
%
%
\begin{thm}
Suppose that $n>0$ and $r,p,q>0$ are integers, 
that
\begin{align}
r+p+q+1 \leq n,
\label{eq_sum_right_0}
\end{align}
that
$B_{r+1},\dots,B_n$ is a sequence of real numbers, that
\begin{align}
B_n < \dots < B_{n+1-q} \leq -2 < B_{n-q} < \dots < B_{n+1-p-q} <
-2 \cdot \cos\left( \frac{\pi}{4} \cdot \frac{1}{p+3/2} \right)
\label{eq_sum_right_1}
\end{align}
and that
\begin{align}
-2 \cdot \cos\left( \frac{\pi}{4} \cdot \frac{1}{p+5/2} \right) \leq
B_{n-q-p} < \dots < 
B_{r+1} < 0.
\label{eq_sum_right_2}
\end{align}
Suppose also that $\varepsilon>0$ is the machine precision,
that $B_{r+1},\dots,B_n$ are defined to precision $\varepsilon$,
and that the real numbers
$y_n,y_{n-1},\dots,y_{r+1},y_r$ are evaluated from $B_{r+1},\dots,B_n$
via the recurrence relation 
\begin{align}
& y_n = 1, \nonumber \\
& y_{n-1} = B_n, \nonumber \\
& y_{j-1} = B_j \cdot y_j - y_{j+1},
\label{eq_err_decay_rec}
\end{align}
for $j< n$ (similar to \eqref{eq_err_grow_rec}, but the direction
is reversed).
Then,
\begin{align}
\text{\rm{rel}}(y_{n-j}) \leq j^2 \cdot \varepsilon,
\label{eq_sum_right_3}
\end{align}
for every $j=1,\dots,q$. Also, 
\begin{align}
\text{\rm{rel}}(y_{n-q-j}) \leq (q+2\cdot j)^2 \cdot \varepsilon,
\label{eq_sum_right_4}
\end{align}
for every $j=1,\dots,p$. In addition,
\begin{align}
& \text{\rm{rel}}
\label{eq_sum_right_5}
\begin{pmatrix} y_{j+1} \\ y_j \end{pmatrix}
\leq 81 \cdot p^2 \cdot (q+2 \cdot p)^2 \cdot \varepsilon, \\
\label{eq_sum_right_5a}
& \text{\rm{rel}}(y_j) \leq 18 \cdot l \cdot (k+2 \cdot l)^2 
\cdot \left| \frac{ y_{n-p-q} }{y_j} \right| \cdot \varepsilon,
\end{align}
for every $j=r, \dots, n-q-p-1$. Finally,
\begin{align}
\text{\rm{rel}}\left(
y_n^2 + \dots + y_{r+2}^2
\right) \leq 
162 \cdot p^2 \cdot (q+2\cdot p)^2 \cdot \varepsilon.
\label{eq_sum_right_6}
\end{align}
\label{thm_summary_right}
\end{thm}
\begin{proof}
We define $\tilde{B}_1,\dots$ and $\tilde{x}_1,\dots$ via the formula
\begin{align}
\tilde{B}_j = -B_{n+1-j}
\label{eq_sumr_1}
\end{align}
and
\begin{align}
\tilde{x}_j = (-1)^{j+1} \cdot y_{n+1-j},
\label{eq_sumr_2}
\end{align}
for $j \geq 1$. Then, due to the combination of \eqref{eq_sumr_1},
\eqref{eq_sumr_2} with \eqref{eq_err_decay_rec},
\begin{align}
\tilde{x_{j+1}} & \; = (-1)^j \cdot y_{n-j} \nonumber \\
 & \; = (-1)^j \cdot \left( B_{n-(j-1)} \cdot y_{n-(j-1)} - y_{n-(j-2)}
\right) \nonumber \\
& \; =
(-1)^(j+1) \cdot \left( \tilde{B}_j \cdot \tilde{x}_j \cdot (-1)^{j+1} +
\tilde{x}_{j-1} \cdot (-1)^j \right) \nonumber \\
& \; = \tilde{B}_j \cdot \tilde{x}_j - \tilde{x}_{j-1},
\label{eq_sumr_3}
\end{align}
for $j \geq 2$. We conclude by combining \eqref{eq_sumr_3} with
\eqref{eq_sum_right_0}, \eqref{eq_sum_right_1} and
Theorem~\ref{thm_summary_left} above.
\end{proof}
\begin{thm}
Suppose, in addition to hypotheses of 
Theorems~\ref{thm_summary_left}, \ref{thm_summary_right},
that the $n \times n$ matrix $B$, defined via the formula
\begin{align}
B = 
\begin{pmatrix}
B_1 & 1 &   &   &   &   \\
1 & B_2 & 1 &   &   &   \\ 
  & 1 & B_3 & 1 &   &   \\
  &   & \ddots & \ddots & \ddots & \\
  &   &   & 1 & B_{n-1} & 1 \\
  &   &   &   & 1 & B_n \\
\end{pmatrix},
\label{eq_glue_1}
\end{align}
is singular, that $x_1,\dots,x_r, x_{r+1}$ are those of 
Theorem~\ref{thm_summary_left}, that
$y_r, y_{r+1}, \dots, y_n$ are 
those of Theorem~\ref{thm_summary_right}, that
the real number $s$ is defined via the formula
\begin{align}
s = \frac{ x_r \cdot y_r + x_{r+1} \cdot y_{r+1} }
         { |x_r \cdot y_r + x_{r+1} \cdot y_{r+1}| } \cdot
 \sqrt{ \frac{x_r^2 + x_{r+1}^2}{y_r^2 + y_{r+1}^2} },
\label{eq_glue_2}
\end{align}
and that the vector $z=(z_1,\dots,z_n)^T$ in $\Rc^n$ is defined via the formula
\begin{align}
z = \left(
x_1, \dots, x_r, x_{r+1}, s \cdot y_{r+2}, \dots, 
s \cdot y_n
\right)^T.
\label{eq_glue_3}
\end{align}
Then, $z$ is an eigenvector of $B$ corresponding to the zero eigenvalue.
Moreover,
\begin{align}
\text{\rm{rel}}(s) \leq 81 \cdot \left( (q+2\cdot p)^2 \cdot p^2 +
(k+2 \cdot l)^2 \cdot l^2 \right) \cdot \varepsilon,
\label{eq_glue_4}
\end{align}
and
\begin{align}
\text{\rm{rel}}(z_1^2 + \dots + z_n^2) \leq
243 \cdot \left(
p^2 \cdot (q+2\cdot p)^2 + 
(k+2 \cdot l)^2 \cdot l^2
\right) \cdot \varepsilon.
\label{eq_glue_11}
\end{align}
\label{thm_glue}
\end{thm}
\begin{proof}
Due to Theorem~\ref{thm_main_rec}
in Section~\ref{sec_vectors}, $x_1,\dots,x_{r+1}$ are the first
$r+1$ coordinates of an eigenvector of $B$ corresponding to the zero 
eigenvalue; also, $y_r, \dots, y_n$ are the last $n+1-r$ coordinates
of an eigenvector of $B$ in the same eigenspace. We combine this observation
with Theorem~\ref{thm_main_rec} and \eqref{eq_glue_2} to conclude
that $z$ is the eigenvector in the null-space of $B$ whose first
coordinate is equal to $1$. The inequality \eqref{eq_glue_4} follows
from the combination of \eqref{eq_sum_left_5} and \eqref{eq_sum_right_5}
(in particular, $s$ in \eqref{eq_glue_2} is well defined). We combine
\eqref{eq_glue_4} with \eqref{eq_sum_right_6} to obtain
\begin{align}
& \text{\rm{rel}}\left(
s^2 \cdot \left(y_n^2 + \dots + y_{r+2}^2
\right) \right) \leq \nonumber \\
& \left(
3 \cdot 81 \cdot p^2 \cdot (q+2\cdot p)^2 + 
81 \cdot (k+2 \cdot l)^2 \cdot l^2 \right)
\cdot \varepsilon.
\label{eq_glue_10}
\end{align}
Finally, we combine \eqref{eq_glue_10} with \eqref{eq_sum_left_6} to obtain
\eqref{eq_glue_11}.
\end{proof}
\begin{cor}
Suppose that, in addition to the hypothesis of Theorem~\ref{thm_glue},
the vector $X \in \Rc^n$ is evaluated from $z$
in \eqref{eq_glue_3} via the formula
\begin{align}
X = (X_1, \dots, X_n)^T = \frac{z}{\| z \|}.
\label{eq_cor_glue_1}
\end{align}
Then,
\begin{align}
\text{\rm{rel}}(X_1) \leq 
243 \cdot \left(
p^2 \cdot (q+2\cdot p)^2 + 
(k+2 \cdot l)^2 \cdot l^2
\right) \cdot \varepsilon,
\label{eq_cor_glue_2}
\end{align}
where $k,l,p,q,r$ are those of 
Theorems~\ref{thm_summary_left}, \ref{thm_summary_right}. More generally,
\begin{align}
\text{\rm{rel}}(X_j) \leq \text{\rm{rel}}(X_1) +
\text{\rm{rel}}(x_j),
\label{eq_cor_glue_3}
\end{align}
for every $2 \leq j \leq r+1$, and
\begin{align}
\text{\rm{rel}}(X_j) \leq \text{\rm{rel}}(X_1) +
\text{\rm{rel}}(y_j) + \text{\rm{rel}}(s),
\label{eq_cor_glue_4}
\end{align}
for every $j=r+2,\dots,n$, where the sequences $\left\{x_j\right\}$, 
$\left\{y_j\right\}$ and the real number $s$ are those
from Theorems~\ref{thm_summary_left}, \ref{thm_summary_right}, \ref{thm_glue}.
\label{cor_glue}
\end{cor}

%
%
\subsection{Asymptotic Error Analysis of a Special Case}
\label{sec_asym}
The analysis of Section~\ref{sec_about_errors}
(e.g. Theorems~\ref{thm_summary_left}, \ref{thm_summary_right},
\ref{thm_glue} and Corollary~\ref{cor_glue}) is carried out for a fairly
general class of sequences $\left\{ B_j \right\}$
(and related matrices $B$ defined via \eqref{eq_glue_1}).
The resulting upper bounds on relative errors of coordinates
of the null-space eigenvector of $B$ depend on the parameters
$k,l,p,q$ determined from $\left\{ B_j \right\}$ via
\eqref{eq_sum_left_1}, \eqref{eq_sum_left_2},
\eqref{eq_sum_right_1}, \eqref{eq_sum_right_2}
(see e.g. the bounds in \eqref{eq_glue_11}, \eqref{eq_cor_glue_2}).

Despite the fact that these bounds are explicitly defined
by $B$, the relation between the {\it relative error} of, say,
the first coordinate $X_1$ of an eigenvector of unit norm
and the {\it magnitude} of $X_1$ is not immediately obvious
(see \eqref{eq_cor_glue_2}). In this section,
this relation is investigated in some detail for a special,
but still fairly broad class of matrices $B$
(that also appear in various applications; see e.g. 
Section~\ref{sec_applications}). 
First, we need
a technical theorem.
\begin{thm}
Suppose that $a \geq 1$ is a real number,
that $\delta > 1$ is a real number,
that the real number $D_a$ is defined via the formula
\begin{align}
D_a =
\sqrt{2} \cdot \int_0^{\pi/2} \left(\sin(\theta)\right)^{1+2/a} \; d\theta,
\label{eq_aa_01}
\end{align}
and that the real number
$\alpha(a,\delta)$ is the solution of the equation
\begin{align}
\alpha^2 \cdot \left( (1+\alpha)^a - 1 \right) \cdot \delta^2
= \frac{\pi^2}{32}
\label{eq_aa_27}
\end{align}
in the unknown $\alpha$.
Then,
\begin{align}
\frac{2 
\cdot \sqrt{2}}{3} \leq D_a 
= \sqrt{\frac{\pi}{2}} \cdot \frac{\Gamma(1+1/a)}{\Gamma(3/2+1/a)},
\leq \sqrt{2},
\label{eq_aa_14}
\end{align}
where $\Gamma$ is the standard Gamma function,
and also
\begin{align}
\alpha(a,\delta) \leq \left( \frac{ \pi^2 }{32 \cdot a \cdot \delta^2 }
\right)^{1/3}.
\label{eq_aa_28}
\end{align}
\label{thm_dalpha}
\end{thm}
\begin{proof}
The proof is straightforward, elementary, and will be omitted.
\end{proof}
The rest of this section is dedicated to asymptotic error analysis
pertaining to a certain class of symmetric tridiagonal matrices.
\begin{thm}
Suppose that $a \geq 1$ is a real number, that $\delta>1$ is a real number,
that the real numbers $D_a, \alpha(a,\delta)$ are those
of Theorem~\ref{thm_dalpha} above.
Suppose also that,
for any real number $c \geq 1$, the real number $\kappa(c)$ is defined
via the formula
\begin{align}
\kappa(c) = \delta^{2/(a+2)} \cdot c^{a/(a+2)},
\label{eq_aa_02}
\end{align}
and the sequence $B_1(c), B_2(c), \dots$ is defined via the formula
\begin{align}
B_j(c) = 2 +2 \cdot \left( \frac{\kappa(c)}{c} \right)^a  
-2 \cdot \left( \frac{j}{c} \right)^a,
\label{eq_aa_1}
\end{align}
for every $j=1,2,\dots$. Suppose also that,
for any real number $c \geq 1$, the sequence $x_1(c), x_2(c), \dots$
is defined from $\left\{ B_j(c) \right\}$ via \eqref{eq_err_grow_rec},
and the integers $k=k(c), l=l(c)$ are defined from 
$\left\{ B_j(c) \right\}$ via \eqref{eq_sum_left_1}, \eqref{eq_sum_left_2}.
Then,
\begin{align}
\label{eq_aa_03}
& k=k(c) = \kappa(c) \cdot (1+o(c)), \quad c \to \infty, \\
\label{eq_aa_04}
& l=l(c) = \alpha(a,\delta) \cdot \kappa(c) \cdot (1+o(c)), 
\quad c \to \infty, 
\end{align}
and also
\begin{align}
x_1(c) \leq x_k(c) \cdot \exp \left(-D_a \cdot \delta
\cdot (1 + o(1)) \right), \quad c \to \infty.
\label{eq_aa_05}
\end{align}
\label{thm_aa}
\end{thm}
\begin{proof}
In this proof, we omit the dependence of various parameters on $c$
whenever it causes no confusion.
First,
\eqref{eq_aa_03} follows from the combination of \eqref{eq_aa_1},
\eqref{eq_aa_02} and \eqref{eq_sum_left_1}.
We substitute \eqref{eq_aa_1}, \eqref{eq_aa_03} into 
\eqref{eq_vec_grow_cor} 
to obtain
\begin{align}
\frac{x_k}{x_1} & \; \geq \prod_{j=2}^k \left(
1 +  \left( \frac{\kappa}{c} \right)^a  
-\left( \frac{j}{c} \right)^a + 
\sqrt{ \left(    
1+
\left( \frac{\kappa}{c} \right)^a  
-\left( \frac{j}{c} \right)^a
\right)^2-1 }
\right) \nonumber \\
& \; = \prod_{j=2}^k \left(1 + \sqrt{ 2 \cdot \left(
\left( \frac{k}{c} \right)^a  
-\left( \frac{j}{c} \right)^a
\right)}
\right) \cdot (1 + o(1)), \quad c \to \infty.
\label{eq_aa_3}
\end{align}
We define the real-valued function $g$ via the formula
\begin{align}
g(x) = 1 +
\sqrt{ 2 \cdot \left(
\left( \frac{k}{c} \right)^a  
-\left( \frac{x}{c} \right)^a
\right)},
\label{eq_aa_4}
\end{align}
for real $0 \leq x \leq k$,
and combine \eqref{eq_aa_1},
\eqref{eq_aa_02}, \eqref{eq_aa_3},
\eqref{eq_aa_4} to obtain
\begin{align}
\prod_{j=2}^k \left(
\frac{B_j}{2} + \sqrt{ \left(\frac{B_j}{2}\right)^2-1 }
\right)
=
\exp \left((1+o(1)) \cdot 
\int_0^k \log(g(x)) \; dx
\right), \quad c \to \infty.
\label{eq_aa_5}
\end{align}
Since $\log(g(k))=0$ due to \eqref{eq_aa_4},
\begin{align}
\int_0^k \log(g(x)) = - \int_0^k x \cdot \frac{d}{dx} \log(g(x)) \; dx
= -\int_0^k \frac{x \cdot g'(x)}{g(x)} \; dx.
\label{eq_aa_6}
\end{align}
We combine \eqref{eq_aa_4} and \eqref{eq_aa_6} to obtain
\begin{align}
\int_0^k \log(g(x)) = 
\frac{a}{\sqrt{2}} \int_0^k \frac{ x^a \; dx }
{
\sqrt{2} \cdot (k^a - x^a) + \sqrt{c^a} \cdot \sqrt{k^a-x^2}
}.
\label{eq_aa_7}
\end{align}
We perform the changes of variable
\begin{align}
x^a = k^a \cdot \sin^2(\theta),
\label{eq_aa_10}
\end{align}
and substitute \eqref{eq_aa_10} into \eqref{eq_aa_7} to obtain
\begin{align}
\int_0^k \log(g(x))
=
k \cdot \int_0^{\pi/2} \frac{ \left(\sin(\theta)\right)^{1+2/a} \; d\theta }
{ \cos(\theta) + \sqrt{c^a/(2 \cdot k^a)}}.
\label{eq_aa_11}
\end{align}
Due to the combination of \eqref{eq_aa_11} and 
\eqref{eq_aa_01}, \eqref{eq_aa_02},
\eqref{eq_aa_03},
\begin{align}
\int_0^k \log(g(x)) = 
D_a \cdot \sqrt{ \frac{k^{a+2} }{ c^a } } \cdot (1+o(1)), \quad c \to \infty,
\label{eq_aa_15}
\end{align}
and we substitute \eqref{eq_aa_15} into \eqref{eq_aa_5} to obtain
\begin{align}
\prod_{j=2}^k \left(
\frac{B_j}{2} + \sqrt{ \left(\frac{B_j}{2}\right)^2-1 }
\right)
= \exp \left(
D_a \cdot \sqrt{ \frac{k^{a+2} }{ c^a } } \cdot (1+o(1))
\right), \quad c \to \infty.
\label{eq_aa_16}
\end{align}
We combine \eqref{eq_aa_16} with \eqref{eq_aa_02}, \eqref{eq_aa_03}
to obtain \eqref{eq_aa_05}.
Next, we combine \eqref{eq_sum_left_1}, \eqref{eq_sum_left_2},
\eqref{eq_aa_02}, \eqref{eq_aa_1}
to obtain
\begin{align}
\frac{(k+l)^a - k^a}{c^a} = 
\frac{\pi^2}{32 \cdot l^2} \cdot (1+o(1)), \quad c \to \infty.
\label{eq_aa_21}
\end{align}
If
\begin{align}
k(c) \ll l(c), \quad c \to \infty,
\label{eq_aa_22}
\end{align}
then due to \eqref{eq_aa_21}
\begin{align}
l^{a+2} = c^a \cdot \frac{ \pi^2 }{32} \cdot (1+o(1)), \quad c \to \infty,
\label{eq_aa_23}
\end{align}
in contradiction to the combination of 
\eqref{eq_aa_22} and \eqref{eq_sum_left_1}, \eqref{eq_sum_left_2}.
If, on the other hand,
\begin{align}
l \ll k, \quad c \to \infty,
\label{eq_aa_24}
\end{align}
then due to \eqref{eq_aa_21}, \eqref{eq_aa_03}
\begin{align}
l^3 = \frac{c^a}{k^{a-1}} \cdot (1+o(c)) = 
O\left( c^{a - (a-1)\cdot a/(a+2)} \right) =
O\left( c^{3a/(a+2)} \right), \quad c \to \infty,
\label{eq_aa_25}
\end{align}
in contradiction to the combination of
\eqref{eq_aa_24} and \eqref{eq_sum_left_1}, \eqref{eq_sum_left_2}. Therefore,
\begin{align}
l(c) = O(k(c)), \quad c \to \infty,
\label{eq_aa_26}
\end{align}
and we combine \eqref{eq_aa_26} with \eqref{eq_aa_02}, \eqref{eq_aa_03},
\eqref{eq_aa_21}
to obtain \eqref{eq_aa_04}.
\end{proof}
The following theorem compliments Theorem~\ref{thm_aa} above.
\begin{thm}
Suppose that $a \geq 1$ and $\varepsilon>0$ are real numbers.
Suppose also that, for any real number $c\geq 1$, the real numbers
$\mu(c), \nu(c), \rho(c)$ are defined via the formulae
\begin{align}
\label{eq_aa_30}
& \mu(c) = \left( \frac{2^{1/a} \cdot c}{a} \right)^{1/3} \cdot
\left( -\frac{3}{4} \cdot \log(\varepsilon) \right)^{2/3}, \\
\label{eq_aa_31}
& \nu(c) = 2^{1/a} \cdot c + \mu(c), \\
\label{eq_aa_32}
& \rho(c) = \left( \frac{ \pi^2 \cdot 2^{1/a} }{ 64 \cdot a } \right)^{1/3}
\cdot c^{1/3},
\end{align}
and that the integer $n(c)$ is defined via the formula
\begin{align}
n(c) = \text{\rm{floor}}(\nu(c)) + 1.
\label{eq_aa_33}
\end{align}
Suppose furthermore that, for any real $c \geq 1$, the sequence
$B_1(c), B_2(c), \dots, $ is defined via \eqref{eq_aa_1},
that the integers $q=q(c)$ and $p=p(c)$ are defined from
$\left\{ B_j(c) \right\}$ via \eqref{eq_sum_right_1}, \eqref{eq_sum_right_2},
and that the sequence $y_1(c),\dots,y_n(c)$ is defined
via \eqref{eq_err_decay_rec}.
Then,
\begin{align}
\label{eq_aa_34}
q(c) = \mu(c) \cdot (1+o(1)), \quad c \to \infty, \\
\label{eq_aa_35}
p(c) = \rho(c) \cdot (1+o(1)), \quad c \to \infty,
\end{align}
and also
\begin{align}
| y_{n(c)}(c) | \leq \varepsilon \cdot |y_{n(c)+1-q(c)}(c)| \cdot (1+o(1)),
\quad c \to \infty.
\label{eq_aa_36}
\end{align}
\label{thm_aaa}
\end{thm}
\begin{proof}
We observe that, due to \eqref{eq_aa_02}, \eqref{eq_sum_right_1},
\begin{align}
2 + 2 \cdot \left( \frac{\kappa(c)}{c} \right)^a - 
2 \cdot \left( \frac{n(c)-q(c)}{c} \right)^a = -2 + o(1), \quad c \to \infty,
\label{eq_aa_41}
\end{align}
and combine \eqref{eq_aa_41}, \eqref{eq_aa_02}, \eqref{eq_aa_03},
\eqref{eq_aa_30}, \eqref{eq_aa_31}, \eqref{eq_aa_33}  to obtain
\eqref{eq_aa_34}.
We combine \eqref{eq_aa_32}, \eqref{eq_aa_33}, \eqref{eq_aa_34},
\eqref{eq_sum_right_1}, \eqref{eq_sum_right_2} to obtain
\begin{align}
\frac{(n-q)^a - (n-q-p)^a}{c^a} & \; =
2 \left( 1 - \left( 1 - \frac{p}{c \cdot 2^{1/a}}\right)^a \right) \cdot
(1+o(1)) \nonumber \\
& \; =
\frac{\pi^2}{32 \cdot p^2} \cdot (1+o(1)), \quad c \to \infty.
\label{eq_aa_43}
\end{align}
We combine \eqref{eq_aa_43} with \eqref{eq_aa_32} to obtain 
\eqref{eq_aa_35}.
Next, for $j=1,\dots,q(c)$,
\begin{align}
B_{n-q+j} & \; = -2 \cdot \left(1 + 
\left( \frac{n(c)-q(c)}{c} \right)^a
\cdot \left( 
\left(1 + \frac{j}{n(c)-q(c)}\right)^a-1
\right)
\right) \nonumber \\
& \; = -2 \cdot \left(1 + \frac{2 \cdot a \cdot j}{2^{1/a} \cdot c} \right)
\cdot (1+o(1)), \quad c \to \infty,
\label{eq_aa_50}
\end{align}
and hence, similar to \eqref{eq_aa_5},
\begin{align}
& \prod_{j=1}^q \left(
\frac{B_{n-q+j}}{2} + \sqrt{ \left(\frac{B_{n-q+j}}{2}\right)^2-1 }
\right)
= \nonumber \\
& \exp \left((1+o(1)) \cdot 
\int_0^q \log\left(
1 + \sqrt{ \frac{4 \cdot a \cdot x}{2^{1/a} \cdot c}}
\right) \; dx
\right), \quad c \to \infty.
\label{eq_aa_51}
\end{align}
We observe that
\begin{align}
\int_0^1 \log(1 + Z \cdot \sqrt{s}) \; ds = \frac{2\cdot Z}{3}
\cdot (1+o(1)), \quad
Z \to 0,
\label{eq_aa_52}
\end{align}
and combine \eqref{eq_aa_51}, \eqref{eq_aa_51} and
Theorem~\ref{thm_vec_decay} in Section~\ref{sec_vectors} to obtain
\begin{align}
| y_n | \leq |y_{n-q+1}| \cdot \exp\left( - \frac{4}{3} \cdot
\sqrt{ \frac{a \cdot q^3}{2^{1/a} \cdot c} } \cdot (1+o(1))
\right), \quad c \to \infty,
\label{eq_aa_53}
\end{align}
and combine \eqref{eq_aa_30}, \eqref{eq_aa_34}, \eqref{eq_aa_53}
to obtain
\eqref{eq_aa_36}.
\end{proof}
The following theorem is a consequence of 
Theorems~\ref{thm_aa}, \ref{thm_aaa} above.
\begin{thm}
Suppose that $\varepsilon>0$ is the machine precision,
and that $a \geq 1$ and $1 \leq \tilde{\delta} < \delta$
are real numbers.
Suppose also that, for any real $c \geq 1$, we define
$\mu(c)$ via \eqref{eq_aa_30}, that $n(c)$ is an integer, that
\begin{align}
2^{1/a} \cdot c < n(c) < 2^{1/a} \cdot c + \mu(c) + 1,
\label{eq_ac_1}
\end{align}
that
the sequence $A_1(c), \dots, A_{n(c)}(c)$ is defined via the formula
\begin{align}
A_j(c) = 2 + 2 \cdot \left( \frac{j}{c} \right)^a,
\label{eq_ac_2}
\end{align}
for every $j=1,\dots,n(c)$, and the $n(c) \times n(c)$ matrix $A(c)$ 
is defined
from $\left\{ A_j(c) \right\}$ via \eqref{eq_running}.
Suppose also that, for any $c \geq 1$, the real number $\lambda(c)$
is an eigenvalue of $A(c)$, that $\delta(c)$ is
a real number, that
\begin{align}
1 < \tilde{\delta} < \delta(c) < \delta,
\label{eq_ac_3}
\end{align}
that
\begin{align}
\lambda(c) = 4 + 2 \cdot \left( \frac{\delta(c)}{c} \right)^{2a/(a+2)},
\label{eq_ac_4}
\end{align}
and that $X(c)=(X_1(c), \dots, X_n(c))^T$
is the unit-norm $\lambda(c)$-eigenvector of $A(c)$.
Suppose furthermore that, for any $c \geq 1$,
the quantities $A_j(c)- \lambda(c)$ are defined to precision $\varepsilon$,
for any $c \geq 1$ and every $j=1,\dots,n(c)$.
Then,
\begin{align}
|X_1(c)| < \exp\left( -\tilde{\delta} \cdot D_a \right)
\cdot (1+o(1)), \quad c \to \infty,
\label{eq_ac_5}
\end{align}
where $D_a$ is defined via \eqref{eq_aa_01}.
Also, if $a>1$, then
\begin{align}
\text{\rm{rel}}(X_1(c)) \leq 620 \cdot \delta^{(16-4a)/(3a+6)} 
\cdot c^{4a/(a+2)}
\cdot \varepsilon
\cdot (1+o(1)), \quad c \to \infty.
\label{eq_ac_6}
\end{align}
If $a=1$, then
\begin{align}
\text{\rm{rel}}(X_1(c)) \leq 960 \cdot \left(
\frac{\delta^{4/3}}{4} + \left(-\log \varepsilon\right)^{4/3} + 1
\right) \cdot c^{4/3} \cdot \varepsilon
\cdot (1+o(1)), \quad c \to \infty.
\label{eq_ac_7}
\end{align}
\label{thm_ac}
\end{thm}
\begin{proof}
Suppose that $c \geq 1$, and
that $k,l,p,q$ are defined from $A(c)$ via
\eqref{eq_sum_left_1}, \eqref{eq_sum_left_2}, \eqref{eq_sum_right_1},
\eqref{eq_sum_right_2}, respectively. If $a>1$, we combine \eqref{eq_ac_1},
\eqref{eq_ac_2}, \eqref{eq_ac_3}, \eqref{eq_ac_4}
with Theorems~\ref{thm_aa}, \ref{thm_aaa} above to obtain
\begin{align}
& 243 \cdot l^2 \cdot (k + 2 \cdot l)^2 = \nonumber \\
& 243 \cdot k^4 \cdot \alpha^2 \cdot (1 + 2\cdot \alpha)^2 < \nonumber \\
&
243 \cdot \left( \frac{\pi^2}{32 \cdot a} \right)^{2/3} \cdot
\left(1 + 2 \cdot \left( \frac{\pi^2}{32 \cdot a} \right) \right) \cdot
\delta^{8/(a+2)-4/3} \cdot c^{4a/(a+2)} <
\nonumber \\
& 620 \cdot \delta^{(4/3) \cdot (4-a)/(a+2)} \cdot
c^{4a/(a+2)}.
\label{eq_ac_10}
\end{align}
and combine \eqref{eq_ac_10} with Corollary~\ref{cor_glue} 
in Section~\ref{sec_about_errors} to obtain
\eqref{eq_ac_6}.
If $a=1$, then
we combine \eqref{eq_ac_1},
\eqref{eq_ac_2}, \eqref{eq_ac_3}, \eqref{eq_ac_4}
with Theorems~\ref{thm_aa}, \ref{thm_aaa} above to obtain
\begin{align}
& 243 \cdot \left( l^2 \cdot (k + 2 \cdot l)^2 + p^2 \cdot (q + 2 \cdot p)^2
\right) \leq \nonumber \\
& 243 \cdot c^{4/3} \cdot \left( \frac{\pi^2}{32} \right)^{2/3} \cdot \left(
\left( \delta^{2/3} + 2 \cdot \left( \frac{\pi^2}{32} \right)^{1/3}
\right)^2 + 
\left( (-3 \cdot \log \varepsilon)^{2/3} + 
2 \cdot \left( \frac{\pi^2}{32} \right)^{1/3}
\right)^2 \right) \nonumber \\
& 960 \cdot c^{4/3} \cdot
\left(
\frac{\delta^{4/3}}{4} + \left(-\log \varepsilon\right)^{4/3} +
1
\right),
\label{eq_ac_11}
\end{align}
and combine \eqref{eq_ac_11} with Corollary~\ref{cor_glue} 
in Section~\ref{sec_about_errors} to obtain
\eqref{eq_ac_7}. For any $a \geq 1$, the inequality
\eqref{eq_ac_5} follows now from \eqref{eq_aa_05}.
\end{proof}
\begin{remark}
The conclusions of Theorem~\ref{thm_ac} above hold even under a
milder assumption that each of $A_j(c)$ and $\lambda(c)$ separately
is defined
to relative precision $\varepsilon$ for every $j$
(and not necessarily their difference).
The related analysis (beyond the scope of this paper) is based on
Theorems~\ref{thm_summary_left}, \ref{thm_summary_right}
in Section~\ref{sec_about_errors}, and on 
the observation that when $\lambda(c) \approx A_j(c)$ what matters
is the absolute (and not relative) accuracy of $\lambda(c)-A_j(c)$.
\label{rem_dif}
\end{remark}

\section{Numerical Algorithms}
\label{sec_num_algo}
In this section, we describe several numerical algorithms
for the evaluation of the eigenvectors
of certain symmetric tridiagonal matrices.

\subsection{Problem Settings}
\label{sec_settings}
%

Suppose that $n > 0$ is an integer,
that $2 < A_1 < A_2 < \dots$ is a sequence of positive
real numbers, that $A$ is an $n$ by $n$ symmetric tridiagonal matrix
defined via \eqref{eq_running} in Section~\ref{sec_vectors},
and that the real number $\lambda$ is an eigenvalue of $A$.

{\bf Task.} Evaluate the unit-length 
eigenvector 
\begin{align}
X=(X_1, \dots, X_n)\in\Rc^n
\label{eq_big_x}
\end{align}
of $A$ corresponding to $\lambda$.

{\bf Desired accuracy of the solution.}
We want the coordinates $X_j$ of $X$ to be evaluated to
high {\it relative} accuracy
(as opposed to {\it absolute} accuracy; see also Section~\ref{sec_intro}).

{\bf Observation.}
This task is potentially difficult if $|X_j|$ 
is small compared to $\| X \|=1$. For example, if $|X_1| < \varepsilon$,
where $\varepsilon$ is the machine precision 
(e.g. $\varepsilon \approx 10^{-16}$
for double-precision calculations),
it is not obvious why $X_1$ should be evaluated to any correct digit
at all (see also Section~\ref{sec_intro}).

{\bf Observation.} 
Due to Theorem~\ref{thm_main_rec} in Section~\ref{sec_vectors},
\begin{align}
X_{j-1} + \left(A_j-\lambda\right) \cdot X_j + X_{j+1} = 0,
\label{eq_algo_main_rec}
\end{align}
for every $j=2,\dots,n-1$. Qualitatively, the relation
between $X_{j-1}, X_j, X_{j+1}$ depends
on whether $\left(A_j-\lambda\right)$ is greater than 2, 
is less than -2, or is between -2 and 2 (see Section~\ref{sec_vectors}).

{\bf Assumption on $\lambda$.}
For the sake of clarity of presentation,
in the rest of this section we assume that the eigenvalue $\lambda$
satisfies the inequality
\begin{align}
2 + A_1 < \lambda < A_n - 2.
\label{eq_algo_lambda_ineq}
\end{align}
Clearly, the obvious simplification of the algorithm 
described below will handle any eigenvalue $\lambda$ of $A$.

\subsection{Informal Description of the Algorithm}
\label{sec_informal}
This section contains an informal description
of an algorithm for the evaluation
of $X=(X_1, \dots, X_n) \in \Rc^n$ 
(see \eqref{eq_big_x}).
On the other hand, Section~\ref{sec_short} below contains a
complete outline
of the steps of the algorithm.

Suppose that $1 < r < n$ is an integer, and that
\begin{align}
A_r \leq \lambda < A_{r+1}
\label{eq_algo_01}
\end{align}
(see \eqref{eq_sum_left_2}, \eqref{eq_sum_right_2}).
For any $\lambda-$eigenvector $x=(x_1, \dots, x_n)\in\Rc^n$
of $A$ and every $j=2,\dots,n-1$, the three consecutive coordinates 
$x_{j-1}, x_j, x_{j+1}$ satisfy the recurrence relation \eqref{eq_main_rec}
of Theorem~\ref{thm_main_rec} (see also \eqref{eq_algo_main_rec}
above).

We set $x_1=1$ and use \eqref{eq_main_rec} to iteratively
evaluate $x_2, \dots, x_{r+1}$ (e.g. "going forward").
Obviously, we have evaluated the first $r+1$ coordinates of $X$
up to a scaling constant. Next, we set $y_n=1$ and use
\eqref{eq_main_rec} to iteratively evaluate
$y_{n-1}, y_{n-2}, \dots, y_r$ (e.g. "going backward").
Again, this gives the last $n-r+1$ coordinates of $X$
up to a {\it different} scaling constant. The accuracy
of both evaluations is investigated in detail in 
Section~\ref{sec_about_errors}.

The indices of the two
sequences overlap at $j=r, r+1$. In exact arithmetic, the planar
vectors
$(x_r,x_{r+1})$ and $(y_r,y_{r+1})$ are linearly dependent
(see Theorem~\ref{thm_glue}
in Section~\ref{sec_about_errors}). We "glue the two sequences together" by
multiplying $y_r, \dots, y_n$ through by the correct scaling factor $s$;
in particular, $x_j = s \cdot y_j$ for $j=r,r+1$.
The resulting vector $z$ in $\Rc^n$ is a $\lambda-$eigenvector of $A$
(see Theorem~\ref{thm_glue}). We then normalize it to obtain $X$.

\subsection{Short Description of the Algorithm}
\label{sec_short}

Suppose that $n>0$ is an integer, that the $n$ by $n$ matrix $A$
is that from Section~\ref{sec_settings}, that $\lambda$
is an eigenvalue of $A$, and that
the integer $1 < r < n$ is 
defined via \eqref{eq_algo_01} above.

\vspace{0.1in}
{\bf Step A:} evaluation of the left coordinates of $X$
(see \eqref{eq_big_x}).

\vspace{0.1in}
{\bf 1.} Set $x_1 = 1$.

{\bf 2.} Compute $x_2$ via \eqref{eq_main_rec_first} 
         of Theorem~\ref{thm_main_rec}.

{\bf 3.} Compute $x_3, \dots, x_r, x_{r+1}$ iteratively 
         via \eqref{eq_main_rec} of Theorem~\ref{thm_main_rec}.

\vspace{0.2in}
{\bf Step B:} evaluation of the right coordinates of $X$.

\vspace{0.1in}
{\bf 1.} Set $y_n=1$.

{\bf 2.} Compute $y_{n-1}$ via \eqref{eq_main_rec_last}
         of Theorem~\ref{thm_main_rec}. 

{\bf 3.} Compute $y_{n-2}, \dots, y_{r+1}, y_r$ iteratively
         via \eqref{eq_main_rec} of Theorem~\ref{thm_main_rec}.

\vspace{0.2in}
{\bf Step C:} glue them together.

\vspace{0.1in}
{\bf 1.} Compute the real number 
$s$ via \eqref{eq_glue_2} in Theorem~\ref{thm_glue}.

{\bf 2.} Compute the vector $z=(z_1,\dots,z_n)$ via \eqref{eq_glue_3}
in Theorem~\ref{thm_glue}.

{\bf 3.} Compute the vector $X=(X_1,\dots,X_n)$ from $z$
via \eqref{eq_cor_glue_1} in Corollary~\ref{cor_glue}.

\vspace{0.1in}
{\bf Observation.} The vector $X \in \Rc^n$ is the unit-norm 
$\lambda-$eigenvector
of $A$ whose first coordinate is positive (see Corollary~\ref{cor_glue}
in Section~\ref{sec_about_errors}).

{\bf Running time.} Obviously, the running time of this algorithm is 
$O(n)$ operations, where $n$ is the dimensionality of the matrix.

\subsection{Accuracy}
\label{sec_error_analysis}
In Sections~\ref{sec_informal}, \ref{sec_short}, we described
an algorithm for the evaluation of the unit length $\lambda-$eigenvector
$X=(X_1, \dots, X_n)$ of $A$, whose first coordinate is positive.
The accuracy of this procedure is investigated in some detail
in Section~\ref{sec_about_errors} for a general tridiagonal matrix
with constant off-diagonal elements and monotone diagonal.
More specifically, the {\it relative} accuracy of various coordinates
is described in 
Theorems~\ref{thm_summary_left}, \ref{thm_summary_right},
\ref{thm_glue} and Corollary~\ref{cor_glue} in Section~\ref{sec_about_errors}.
For example, \eqref{eq_cor_glue_2} provides a bound on $\text{\rm{rel}}(X_1)$
in terms of the integers $1<k,l,p,q<n$ (defined
via \eqref{eq_sum_left_1}, \eqref{eq_sum_left_2},
\eqref{eq_sum_right_1}, \eqref{eq_sum_right_2}) and
the relative accuracy $\varepsilon$ of $\lambda-A_j$ for $j=1,\dots,n$
(see also Remark~\ref{rem_dif} in Section~\ref{sec_asym}).
We summarize the results of Section~\ref{sec_about_errors}
qualitatively in the following observations
(see also Section~\ref{sec_numerical} for related numerical experiments).

{\bf Observation 1.} For all $j$ such that $\lambda-A_j \geq 2$
(e.g. for $1 \leq j \leq k$ in the notation of Theorem~\ref{thm_summary_left}
in Section~\ref{sec_about_errors}),
the coordinates $X_j$ are evaluated to roughly the same {\it relative} accuracy,
{\it independent} of how small they are
(see e.g. Theorem~\ref{thm_err_grow_rel} 
in Section~\ref{sec_about_errors} and
 \eqref{eq_sum_left_3} in Theorem~\ref{thm_summary_left}).
These coordinates form a monotonically
increasing sequence (see Theorem~\ref{thm_vec_grow} 
in Section~\ref{sec_vectors} for an estimate on its growth).

{\bf Observation 2.} For all $j$ such that $\lambda-A_j \leq -2$
(e.g. for $n-q \leq j \leq n$ 
in the notation of Theorem~\ref{thm_summary_right}
in Section~\ref{sec_about_errors}),
the coordinates $X_j$ are evaluated to roughly the same {\it relative} accuracy,
{\it independent} of how small they are
(see e.g.
 \eqref{eq_sum_right_3} in Theorem~\ref{thm_summary_right}).
These coordinates form an alternating sequence, and
their absolute values form a monotonically
decreasing sequence (see Theorem~\ref{thm_vec_decay}
in Section~\ref{sec_vectors} 
for an estimate on its decay).

{\bf Observation 3.} For all $j$ such that 
$\lambda-2 \leq A_j \leq \lambda+2$ (e.g. for $k < j < n-q$ in
the notation of Theorems~\ref{thm_summary_left}, \ref{thm_summary_right})
in Section~\ref{sec_about_errors},
the coordinates $X_j$ are evaluated to roughly the same
{\it absolute} accuracy
(see e.g. \eqref{eq_osc_14} in Theorem~\ref{thm_osc},
\eqref{eq_sum_left_5}, \eqref{eq_sum_left_5a} in Theorem~\ref{thm_summary_left},
\eqref{eq_sum_right_5}, \eqref{eq_sum_right_5a} in 
Theorem~\ref{thm_summary_right}).
These coordinates vary in magnitude in a fairly moderate way
and exhibit an oscillatory behavior
(see e.g. Theorems~\ref{thm_mat}, \ref{thm_delta} and 
Corollaries~\ref{cor_mat}, \ref{cor_delta} 
in Section~\ref{sec_vectors},  
and also Section~\ref{sec_numerical}).

%
\begin{remark}
Extensive numerical experiments seem to indicate that the estimates
from Section~\ref{sec_about_errors}
 are somewhat pessimistic. In other words,
in practice the relative error tends to be smaller than our
estimates suggest (see also Section~\ref{sec_numerical}).
\label{rem_acc}
\end{remark}
\begin{remark}
It is somewhat surprising that,
according to \eqref{eq_cor_glue_3} in Corollary~\ref{cor_glue},
the relative error of, say, $X_1$
seems to be independent of the order of magnitude of $X_1$.
In particular, while $X_1$ can be fairly small
(see e.g. Theorem~\ref{thm_vec_grow} and
Corollary~\ref{cor_vec_grow} in Section~\ref{sec_vectors}), 
it still will be evaluated
to reasonable relative precision.
\label{rem_surprise}
\end{remark}
\begin{remark}
When the coordinates of the eigenvector are evaluated
via the three-terms recurrence \eqref{eq_main_rec},
the choice of direction plays a crucial role.
Roughly speaking,
this recurrence is unstable in the backward direction in the
region of growth, and is unstable in the forward direction
in the region of decay (see also Section~\ref{sec_about_errors}). 
As expected, the use of this recurrence relation in
a "wrong" direction leads to a disastrous loss of accuracy.
\label{rem_simplified_direction}
\end{remark}

\subsection{Related Algorithms}
\label{sec_related}
In Section~\ref{sec_informal}, \ref{sec_short},
we presented an algorithm 
for accurate evaluation of the coordinates of the eigenvector $X$
(see \eqref{eq_big_x} in Section~\ref{sec_settings}).
In this section, we briefly discuss the accuracy of several
classical algorithms for the solution of the same problem.

\subsubsection{Inverse Power}
\label{sec_algo_inverse_power}

The unit-length $\lambda-$eigenvector $X$ of $A$ can be obtained
via Inverse Power Method with Shifts
(see Section~\ref{sec_power} for more details). This method is iterative,
and, on each iteration, the approximation $x^{(k+1)}$ of $X$
is obtained from $x^{(k)}$ via solving the linear system
\begin{align}
\left( \lambda\cdot I - A \right) \cdot x^{(k+1)} = x^{(k)},
\label{eq_power_system}
\end{align}
and normalizing the solution.
We observe that this method also evaluates $\lambda$
(even though in Section~\ref{sec_settings} we assume that
$\lambda$ has already been evaluated).
On each iteration, we solve the linear system \eqref{eq_power_system}
by Gaussian elimination
(since $A$ is tridiagonal, each iteration
costs $O(n)$ operations; moreover, $O(1)$ iterations are required:
see Remark~\ref{rem_power} in Section~\ref{sec_power}).

The following conjecture about the accuracy of Inverse Power Method
is substantiated by extensive numerical experiments
(see Section~\ref{sec_numerical}).
\begin{conjecture}
Suppose that $\varepsilon>0$ is the machine precision
(e.g. $\varepsilon \approx 10^{-16}$
for double-precision calculations), and that the eigenvalue $\lambda$
of $A$ is defined to accuracy $\varepsilon$.
Suppose also that $\lambda-A_1>2$. Suppose furthermore that $K>0$ is an integer,
and that
\begin{align}
K > \frac{\log\left(|X_1|\right)}{\log(\varepsilon)} + 1,
\label{eq_inverse_k}
\end{align}
where $X=(X_1, \dots, X_n) \in \Rc^n$ is the unit-length
$\lambda-$eigenvector of $A$. Then, 
after $K$ iterations of Inverse Power Method,
$X_1$ is evaluated to high relative accuracy. More specifically,
this relative accuracy is roughly of the same order of magnitude as for
the algorithm described in Sections~\ref{sec_informal},
\ref{sec_short} (see also \eqref{eq_exp1_10}, \eqref{eq_exp3_10}
below).
\label{conj_inverse}
\end{conjecture}
\begin{remark}
The inequality \eqref{eq_inverse_k} reflects on the fact
that each iteration of Inverse Power Method can reduce the coordinates
of the approximation $x^{(k)}$ by a factor of at most $\varepsilon^{-1}$.
In other words, if $X_1 \approx 10^{-50}$, and, in the initial approximation,
$x^{(1)}_1 = O(1)$, then $x^{(4)}_1$ will already be of the same order
of magnitude as $X_1$, and $x^{(5)}$ will approximate $X_1$
to a high relative precision.
\label{rem_inverse_k}
\end{remark}

\subsubsection{Jacobi Rotations}
\label{sec_algo_jacobi}
In the view of Section~\ref{sec_algo_inverse_power},
one might suspect that virtually any standard algorithm
would accurately solve the problem introduced in
Section~\ref{sec_settings}. In other words, one might suspect that
the small coordinates of $X$ in the region of growth and
the region of decay will be evaluated to high relative precision
by any reasonable algorithm that computes eigenvectors.

Unfortunately, this is emphatically not the case, and the accuracy
of the result strongly depends on the choice of the algorithm.
For example, the popular Jacobi Rotations
algorithm for the evaluation of the eigenvalues and eigenvectors of 
a symmetric matrix $A$
(see, for example, \cite{Dahlquist}, \cite{Golub}, 
\cite{Stoer_Bul}, \cite{Wilkinson})
typically evaluates the eigenvalues of $A$
fairly accurately. Moreover, the corresponding unit-length eigenvectors are
evaluated to high relative accuracy, in the sense 
of \eqref{eq_mat_intro_rel_acc} in Section~\ref{sec_intro}.
However, the $\emph{coordinates}$ of $X$ are typically evaluated
only to high $\emph{absolute}$ accuracy.
In particular,
the relative accuracy of small coordinates will typically be poor:
if, for example, $X_1 \approx 10^{-50}$,
its numerical approximation, produced by Jacobi Rotations,
will usually have no correct digits at all
(the latest statement is supported by extensive numerical evidence).

\subsubsection{Gaussian Elimination}
\label{sec_gaussian}
Another possible method to evaluate $X$ would be to solve
the linear system
\begin{align}
\left(\lambda \cdot I - A \right) \cdot X = 0,
\label{eq_gauss_system}
\end{align}
by means of Gaussian Elimination
(see, for example, \cite{Dahlquist}, \cite{Golub}, 
\cite{Stoer_Bul}, \cite{Wilkinson}).
Unfortunately, this method, in general,
fails to evaluate the small coordinates of $X$ with high
relative accuracy (see, however, Section~\ref{sec_algo_inverse_power},
where Gaussian Elimination is used several times,
as a step of Inverse Power Method).

\section{Applications}
\label{sec_applications}

In this section, we describe some applications of the algorithm
from Section~\ref{sec_num_algo}
to other computational problems.

\subsection{Bessel Functions}
\label{sec_mat_bessel}

Suppose that $x>0$ is a real number,
and that $m>0$ is an integer. 
Below we describe a connection between
the classical
algorithm for the evaluation of $J_0(x),J_{\pm 1}(x),\dots,J_{\pm}(x)$
from Section~\ref{sec_num_bessel}
and the scheme from Section~\ref{sec_short}.

Suppose that $N>m$ is an integer (see Remark~\ref{rem_bessel_n}
in Section~\ref{sec_num_bessel}),
that the symmetric tridiagonal $(2N+1) \times (2N+1)$ 
matrix $A=A(x)$ is that
of
Theorem~\ref{thm_ac} in Section~\ref{sec_asym} with $a=1$ and $c=x$
(see \eqref{eq_ac_2}),
and that the real number $\lambda$
is defined via the formula
\begin{align}
\lambda = 2 + \frac{2 \cdot (N+1)}{x}.
\label{eq_bessel_num_lambda}
\end{align}
In the notation of Section~\ref{sec_num_bessel}, $\lambda$ is
an eigenvalue of $A$, and the corresponding unit-length eigenvector
$X$ is precisely
\begin{align}
X = \frac{1}{d} \cdot \left(
\tilde{J}_N, \dots, 
\tilde{J}_1, \tilde{J}_0, -\tilde{J}_1, \dots, (-1)^N \cdot \tilde{J}_N
\right).
\label{eq_bessel_num_vec}
\end{align}
In addition, the evaluation $\tilde{J}_0,\dots,\tilde{J}_N$ via the 
scheme described in Section~\ref{sec_num_bessel} 
(see \eqref{eq_bessel_d})
is
essentially identical to the evaluation of $X$ in \eqref{eq_bessel_num_vec}
via the algorithm from Section~\ref{sec_short}.

We conclude that the accuracy of this evaluation
has been analyzed in Theorems~\ref{thm_aa}, \ref{thm_aaa} in
Section~\ref{sec_asym},
and, despite the scheme being classical, this analysis
appears to be new (see \eqref{eq_ac_7} in Theorem~\ref{thm_ac}
in Section~\ref{sec_asym}
and Conjecture~\ref{conj_exp1} in Section~\ref{sec_intro},
as well as Section~\ref{sec_mat_exp3} for the related numerical experiments).

\subsection{Prolate Spheroidal Wave Functions}
\label{sec_mat_pswf_num}

Suppose that $c>0$ is a real number, and that the integral operator
$F_c:L^2[-1,1]\to[-1,1]$ is defined via the formula
\begin{align}
F_c[\varphi](x) = \int_{-1}^1 \varphi(t) \cdot e^{icxt} \; dt.
\label{eq_mat_pswf_fc}
\end{align}
Suppose also that the complex numbers $\lambda_0(c),\lambda_1(c),\dots$
are the eigenvalues of $F_c$ (ordered such that 
$|\lambda_0(c)| > |\lambda_1(c)| > \dots$).
The prolate spheroidal wave functions
(PSWFs)
corresponding to the band limit $c$ 
are the unit-norm eigenfunctions $\psi_0^{(c)}, \psi_1^{(c)}, \dots $ 
of $F_c$
(see e.g. \cite{RokhlinXiaoProlate}, 
\cite{LandauWidom},
\cite{ProlateSlepian1},
\cite{ProlateLandau1},
\cite{prol_book}).

It turns out that, for any $n \geq 0$, the eigenvalue $\lambda_n(c)$
can be evaluated at $O(1)$ operations from the first coordinate
of the unit-length eigenvector corresponding to a certain
eigenvalue of a symmetric tridiagonal matrix $A(c)$; moreover,
this matrix is essentially a perturbed version of
the matrix from Theorem~\ref{thm_ac} in Section~\ref{sec_asym},
with $a=2$ (see e.g. \cite{Report5ACHA}, \cite{prol_book}
for more details).

In particular, the algorithm of 
Sections~\ref{sec_informal}, \ref{sec_short},
with obvious minor modifications,
is applicable to the task of evaluating $\lambda_n(c)$
numerically with high relative accuracy
(even when $|\lambda_n(c)| < \varepsilon$, where $\varepsilon>0$
is the machine precision).
Moreover, the error analysis of such evaluation,
in a somewhat more general form,
has been carried
out in Theorems~\ref{thm_aa}, \ref{thm_aaa}, \ref{thm_ac} in
Section~\ref{sec_asym}
(see also Corollary~\ref{cor_glue} in Section~\ref{sec_about_errors}).

In Section~\ref{sec_numerical}, we present
several related numerical examples.
For the results of additional numerical experiments,
see, for example, \cite{Report5ACHA}.

\section{Numerical Results}
\label{sec_numerical}
In this section,
we illustrate the analysis of Section~\ref{sec_analytical}
via several numerical experiments. All the calculations were
implemented in FORTRAN (the Lahey 95 LINUX version), and were
carried out in
double precision. In addition, extended precision calculations
were used to estimate the accuracy of double precision
calculations.

\subsection{Experiment 1.}
\label{sec_mat_exp1}
In this experiment, we illustrate the performance of the
algorithm on certain matrices. 

{\bf Description.}
We first choose, more or less arbitrarily,
the real numbers $a,\delta \geq 0$. Then, for each choice
of five different values $c=10^2,10^3,10^4,10^5,10^6$,
we proceed as follows. We define the integer $n=n(c)$ via
\eqref{eq_aa_33} in Theorem~\ref{thm_aaa}, 
define $A_1,\dots,A_n$ via \eqref{eq_ac_2} in Theorem~\ref{thm_ac},
and then define
the symmetric tridiagonal $n \times n$ matrix $A=A(c)$ via
\eqref{eq_running}. Then, we define the real number $\tilde{\lambda}$
via the formula
\begin{align}
\tilde{\lambda} = 4 + 2 \cdot \left( \frac{\delta}{c} \right)^{2a/(a+2)},
\label{eq_exp1_01}
\end{align}
(see \eqref{eq_ac_4} in Theorem~\ref{thm_ac}), and find the closest
eigenvalue $\lambda(c)$ of $A(c)$ by Shifted Inverse Power method,
using $\tilde{\lambda}$ as the initial approximation to $\lambda(c)$
(see Section~\ref{sec_power}). We then compute $\delta(c)$ from $\lambda(c)$
via \eqref{eq_ac_4}.

Next, we obtain the unit-length $\lambda(c)$-eigenvector of $A$ by four
different methods:

{\bf 1.} $Y=(Y_1,\dots,Y_n)$ via 30 iterations of
Shifted Inverse Power, in double precision.

{\bf 2.} $X=(X_1,\dots,X_n)$ via the algorithm from Section~\ref{sec_short},
in double precision.

{\bf 3.} $\hat{Y}=(\hat{Y}_1,\dots,\hat{Y}_n)$ 
via 30 iterations of Shifted Inverse Power, in extended precision
(we also recompute the eigenvalue $\hat{\lambda}(c)$ in extended precision).

{\bf 4.} $\hat{X}=(\hat{X}_1,\dots,\hat{X}_n)$ 
via the algorithm from Section~\ref{sec_short},
in extended precision.

We verify that each of $\hat{X}$ and $\hat{Y}$ satisfies the definition of an
eigenvector coordinate-wise 
to at least 17 decimal digits, 
and also that $\hat{X}=\hat{Y}$ to at least 17 decimal digits.
In other words, each of $\hat{X}, \hat{Y}$ is the unit-length
$\lambda(c)-$eigenvector of $A$ defined to full double precision.
We use this observation to evaluate the relative and absolute
errors of $X_j, Y_j$, for every $j=1,\dots,n$.

For every $a=1,2,3,4,6$, we repeat this procedure for ten
different values of $\delta$ between 50 and 200.

{\bf Tables and Figures.}
The results of the experiment are displayed in 
Tables~\ref{t:test513a2_delta50}--\ref{t:test513a4_delta150}.
Each of these tables corresponds to a particular choice
of $a$ and $\delta$, and has the following structure.
Each of five columns corresponds to a different value of $c$,
between $10^2$ and $10^6$. The first three rows contain
$c$, the matrix size $n$, and the index $k$ (such that 
$A_k \approx \lambda(c)-2$: see 
\eqref{eq_sum_left_1} in Theorem~\ref{thm_summary_left}
for the precise definition).
The next two rows contain the eigenvalue $\lambda(c)$ and the
related real number
$\delta(c)$ (see \eqref{eq_exp1_01}). The next two rows contain
the coordinates $X_1$ and $X_k$.
The next two rows contain the relative accuracy of $X_1$ and $Y_1$.
The last two rows contain the maximal absolute accuracy
among all coordinates of $X, Y$, respectively.

\begin{table}[htbp]
\begin{center}
\begin{tabular}{c|c|c|c|c|c}
$c$ & $10^2$ & $10^3$ & $10^4$ & $10^5$ & $10^6$ \\
$n$  & 180 & 1,497 & 14,320 & 141,803 & 1,415,035 \\
$k$  & 71 & 226 & 706 & 2,244 & 7,109 \\
\hline
$\lambda$  & 0.50164E+01 & 0.41021E+01 & 0.40099E+01 & 0.40010E+01 & 0.40000E+01 \\
$\delta(c) $ & 0.50826E+02 & 0.51086E+02 & 0.49906E+02 & 0.50379E+02 & 0.50551E+02 \\
\hline
$X_1$  & 0.19744E-24 & 0.46025E-26 & 0.21813E-26 & 0.20152E-27 & 0.26903E-28 \\
$X_k$  & 0.12621E+00 & 0.60020E-01 & 0.28690E-01 & 0.14439E-01 & 0.73972E-02 \\
\hline
$\text{\rm{rel}}(X_1)$  & 0.19302E-13 & 0.26421E-12 & 0.43114E-11 & 0.13247E-10 & 0.11171E-09 \\
$\text{\rm{rel}}(Y_1)$  & 0.55816E-14 & 0.24161E-13 & 0.55651E-12 & 0.68590E-11 & 0.30212E-10 \\
$\underset{j}{\max}|X_j-\hat{X}_j|$  & 0.17885E-14 & 0.10874E-13 & 0.81497E-13 & 0.11156E-12 & 0.48541E-12 \\
$\underset{j}{\max}|Y_j-\hat{Y}_j|$  & 0.47183E-15 & 0.62991E-15 & 0.86371E-14 & 0.56234E-13 & 0.13395E-12 \\
\end{tabular}
\end{center}
\caption{\it
Experiment 1. Parameters: $a=2$, $\delta=50$.
}
\label{t:test513a2_delta50}
\end{table}
\begin{table}[htbp]
\begin{center}
\begin{tabular}{c|c|c|c|c|c}
$c$ & $10^2$ & $10^3$ & $10^4$ & $10^5$ & $10^6$ \\
$n$  & 180 & 1,497 & 14,320 & 141,803 & 1,415,035 \\
$k$  & 101 & 315 & 1,004 & 3,180 & 9,992 \\
\hline
$\lambda$  & 0.60503E+01 & 0.41993E+01 & 0.40201E+01 & 0.40019E+01 & 0.40002E+01 \\
$\delta(c) $ & 0.10251E+03 & 0.99703E+02 & 0.10087E+03 & 0.10118E+03 & 0.99842E+02 \\
\hline
$X_1$  & 0.29706E-47 & 0.33654E-49 & 0.73691E-51 & 0.77717E-52 & 0.54741E-52 \\
$X_k$  & 0.13199E+00 & 0.56585E-01 & 0.28214E-01 & 0.14026E-01 & 0.71872E-02 \\
\hline
$\text{\rm{rel}}(X_1)$  & 0.14729E-13 & 0.20625E-12 & 0.14676E-11 & 0.40918E-10 & 0.46459E-10 \\
$\text{\rm{rel}}(Y_1)$  & 0.47051E-14 & 0.39500E-13 & 0.57239E-12 & 0.68254E-11 & 0.32697E-10 \\
$\underset{j}{\max}|X_j-\hat{X}_j|$  & 0.11519E-14 & 0.79096E-14 & 0.21711E-13 & 0.30486E-12 & 0.17340E-12 \\
$\underset{j}{\max}|Y_j-\hat{Y}_j|$  & 0.78063E-15 & 0.75123E-15 & 0.77475E-14 & 0.52657E-13 & 0.12385E-12 \\
\end{tabular}
\end{center}
\caption{\it
Experiment 1. Parameters: $a=2$, $\delta=100$.
}
\label{t:test513a2_delta100}
\end{table}
\begin{table}[htbp]
\begin{center}
\begin{tabular}{c|c|c|c|c|c}
$c$ & $10^2$ & $10^3$ & $10^4$ & $10^5$ & $10^6$ \\
$n$  & 180 & 1,497 & 14,320 & 141,803 & 1,415,035 \\
$k$  & 123 & 389 & 1,227 & 3,875 & 12,296 \\
\hline
$\lambda$  & 0.70491E+01 & 0.43029E+01 & 0.40301E+01 & 0.40030E+01 & 0.40003E+01 \\
$\delta(c) $ & 0.15244E+03 & 0.15146E+03 & 0.15076E+03 & 0.15021E+03 & 0.15121E+03 \\
\hline
$X_1$  & 0.24360E-68 & 0.10108E-73 & 0.79506E-75 & 0.19809E-75 & 0.10325E-76 \\
$X_k$  & 0.14129E+00 & 0.59531E-01 & 0.27646E-01 & 0.13861E-01 & 0.70498E-02 \\
\hline
$\text{\rm{rel}}(X_1)$  & 0.57053E-14 & 0.39666E-12 & 0.31336E-13 & 0.28896E-10 & 0.16840E-09 \\
$\text{\rm{rel}}(Y_1)$  & 0.29582E-14 & 0.44484E-13 & 0.58518E-12 & 0.69078E-11 & 0.32768E-10 \\
$\underset{j}{\max}|X_j-\hat{X}_j|$  & 0.64401E-15 & 0.14322E-13 & 0.88880E-14 & 0.19695E-12 & 0.58894E-12 \\
$\underset{j}{\max}|Y_j-\hat{Y}_j|$  & 0.30530E-15 & 0.81206E-15 & 0.75181E-14 & 0.50491E-13 & 0.11008E-12 \\
\end{tabular}
\end{center}
\caption{\it
Experiment 1. Parameters: $a=2$, $\delta=150$.
}
\label{t:test513a2_delta150}
\end{table}
\begin{table}[htbp]
\begin{center}
\begin{tabular}{c|c|c|c|c|c}
$c$ & $10^2$ & $10^3$ & $10^4$ & $10^5$ & $10^6$ \\
$n$  & 148 & 1,251 & 12,025 & 119,207 & 1,189,823 \\
$k$  & 80 & 371 & 1,725 & 8,052 & 37,584 \\
\hline
$\lambda$  & 0.48307E+01 & 0.40378E+01 & 0.40018E+01 & 0.40000E+01 & 0.40000E+01 \\
$\delta(c) $ & 0.51745E+02 & 0.51066E+02 & 0.51375E+02 & 0.52214E+02 & 0.53092E+02 \\
\hline
$X_1$  & 0.15657E-27 & 0.56925E-29 & 0.34307E-30 & 0.11988E-31 & 0.40486E-33 \\
$X_k$  & 0.16156E+00 & 0.70686E-01 & 0.31217E-01 & 0.14289E-01 & 0.65824E-02 \\
\hline
$\text{\rm{rel}}(X_1)$  & 0.24916E-13 & 0.17710E-12 & 0.82978E-11 & 0.45445E-09 & 0.39497E-08 \\
$\text{\rm{rel}}(Y_1)$  & 0.50118E-14 & 0.40247E-13 & 0.15850E-11 & 0.24346E-10 & 0.10033E-09 \\
$\underset{j}{\max}|X_j-\hat{X}_j|$  & 0.14710E-14 & 0.50368E-14 & 0.85255E-13 & 0.21667E-11 & 0.85706E-11 \\
$\underset{j}{\max}|Y_j-\hat{Y}_j|$  & 0.53949E-15 & 0.14180E-14 & 0.15365E-13 & 0.11264E-12 & 0.27496E-12 \\
\end{tabular}
\end{center}
\caption{\it
Experiment 1. Parameters: $a=4$, $\delta=50$.
}
\label{t:test513a4_delta50}
\end{table}
\begin{table}[htbp]
\begin{center}
\begin{tabular}{c|c|c|c|c|c}
$c$ & $10^2$ & $10^3$ & $10^4$ & $10^5$ & $10^6$ \\
$n$  & 149 & 1,251 & 12,025 & 119,207 & 1,189,823 \\
$k$  & 99 & 468 & 2,160 & 10,074 & 46,353 \\
\hline
$\lambda$  & 0.59504E+01 & 0.40964E+01 & 0.40044E+01 & 0.40002E+01 & 0.40000E+01 \\
$\delta(c) $ & 0.98136E+02 & 0.10293E+03 & 0.10085E+03 & 0.10226E+03 & 0.99596E+02 \\
\hline
$X_1$  & 0.65592E-50 & 0.16890E-56 & 0.13441E-56 & 0.22928E-58 & 0.60367E-58 \\
$X_k$  & 0.16663E+00 & 0.69229E-01 & 0.31413E-01 & 0.14323E-01 & 0.65935E-02 \\
\hline
$\text{\rm{rel}}(X_1)$  & 0.36733E-13 & 0.11611E-12 & 0.64777E-11 & 0.56745E-10 & 0.58871E-08 \\
$\text{\rm{rel}}(Y_1)$  & 0.17734E-13 & 0.71711E-13 & 0.15602E-11 & 0.24704E-10 & 0.12500E-09 \\
$\underset{j}{\max}|X_j-\hat{X}_j|$  & 0.20053E-14 & 0.38650E-14 & 0.58993E-13 & 0.24506E-12 & 0.11594E-10 \\
$\underset{j}{\max}|Y_j-\hat{Y}_j|$  & 0.88124E-15 & 0.11261E-14 & 0.14018E-13 & 0.10435E-12 & 0.30450E-12 \\
\end{tabular}
\end{center}
\caption{\it
Experiment 1. Parameters: $a=4$, $\delta=100$.
}
\label{t:test513a4_delta100}
\end{table}
\begin{table}[htbp]
\begin{center}
\begin{tabular}{c|c|c|c|c|c}
$c$ & $10^2$ & $10^3$ & $10^4$ & $10^5$ & $10^6$ \\
$n$  & 148 & 1,251 & 12,025 & 119,207 & 1,189,823 \\
$k$  & 115 & 535 & 2,472 & 11,446 & 53,300 \\
\hline
$\lambda$  & 0.75092E+01 & 0.41649E+01 & 0.40075E+01 & 0.40003E+01 & 0.40000E+01 \\
$\delta(c) $ & 0.15244E+03 & 0.15386E+03 & 0.15112E+03 & 0.14999E+03 & 0.15141E+03 \\
\hline
$X_1$  & 0.28839E-74 & 0.19053E-83 & 0.17930E-83 & 0.66661E-84 & 0.11235E-85 \\
$X_k$  & 0.19676E+00 & 0.68972E-01 & 0.31725E-01 & 0.14354E-01 & 0.66008E-02 \\
\hline
$\text{\rm{rel}}(X_1)$  & 0.76598E-14 & 0.18105E-12 & 0.89870E-11 & 0.47429E-09 & 0.44710E-08 \\
$\text{\rm{rel}}(Y_1)$  & 0.76598E-14 & 0.67667E-13 & 0.13524E-11 & 0.23468E-10 & 0.14141E-09 \\
$\underset{j}{\max}|X_j-\hat{X}_j|$  & 0.25396E-14 & 0.53898E-14 & 0.80525E-13 & 0.19056E-11 & 0.81957E-11 \\
$\underset{j}{\max}|Y_j-\hat{Y}_j|$  & 0.24146E-14 & 0.10780E-14 & 0.12499E-13 & 0.99000E-13 & 0.22676E-12 \\
\end{tabular}
\end{center}
\caption{\it
Experiment 1. Parameters: $a=4$, $\delta=150$.
}
\label{t:test513a4_delta150}
\end{table}
\begin{table}[htbp]
\begin{center}
\begin{tabular}{c|c|c|c|c|c}
$a$  & 1 & 2 & 3 & 4 & 6 \\
\hline
$\beta_Y(a)$  & 0.791E+00 & 0.104E+01 & 0.103E+01 & 0.109E+01 & 0.110E+01 \\
\hline
$\beta_X(a)$  & 0.586E+00 & 0.101E+01 & 0.115E+01 & 0.131E+01 & 0.146E+01 \\
$\beta(a)$  & 0.666E+00 & 0.100E+01 & 0.119E+01 & 0.133E+01 & 0.150E+01 \\
$4a/(a+2)$  & 0.133E+01 & 0.200E+01 & 0.239E+01 & 0.266E+01 & 0.300E+01 \\
\end{tabular}
\end{center}
\caption{\it
Experiment 1. Best fit slopes of $\log_{10}(\text{\rm{rel}}(Y_1))$,
$\log_{10}(\text{\rm{rel}}(X_1))$
as functions of $\log_{10}(c)$.
}
\label{t:test519}
\end{table}
\begin{table}[htbp]
\begin{center}
\begin{tabular}{c|c|c|c|c|c}
$c$ & $10^2$ & $10^3$ & $10^4$ & $10^5$ & $10^6$ \\
$m$  & 162 & 1,135 & 10,292 & 100,629 & 1,001,357 \\
$N$  & 192 & 1,175 & 10,392 & 100,829 & 1,001,757 \\
\hline
$J_m(c)$  & 0.13298E-20 & 0.11471E-21 & 0.32071E-22 & 0.14301E-22 & 0.59576E-23 \\
$J_c(c)$  & 0.96366E-01 & 0.44730E-01 & 0.20762E-01 & 0.96369E-02 & 0.44730E-02 \\
\hline
$|1-X_m/J_m(c)|$  & 0.33801E-13 & 0.15085E-12 & 0.24630E-12 & 0.22284E-11 & 0.77524E-11 \\
$|1-Y_m/J_m(c)|$  & 0.36770E-14 & 0.22545E-13 & 0.14788E-12 & 0.98237E-12 & 0.24681E-11 \\
\end{tabular}
\end{center}
\caption{\it
Experiment 3. Parameters: $\delta=50$.
}
\label{t:test515_delta50}
\end{table}
\begin{table}[htbp]
\begin{center}
\begin{tabular}{c|c|c|c|c|c}
$c$ & $10^2$ & $10^3$ & $10^4$ & $10^5$ & $10^6$ \\
$m$  & 200 & 1,215 & 10,464 & 101,000 & 1,002,154 \\
$N$  & 230 & 1,255 & 10,564 & 101,200 & 1,002,554 \\
\hline
$J_m(c)$  & 0.20593E-40 & 0.61117E-42 & 0.10612E-42 & 0.39770E-43 & 0.18323E-43 \\
$J_c(c)$  & 0.96366E-01 & 0.44730E-01 & 0.20762E-01 & 0.96369E-02 & 0.44730E-02 \\
\hline
$|1-X_m/J_m(c)|$  & 0.38368E-14 & 0.13658E-12 & 0.20091E-11 & 0.10091E-11 & 0.56160E-11 \\
$|1-Y_m/J_m(c)|$  & 0.28466E-14 & 0.93836E-14 & 0.14805E-12 & 0.11176E-11 & 0.29720E-11 \\
\end{tabular}
\end{center}
\caption{\it
Experiment 3. Parameters: $\delta=100$.
}
\label{t:test515_delta100}
\end{table}
\begin{table}[htbp]
\begin{center}
\begin{tabular}{c|c|c|c|c|c}
$c$ & $10^2$ & $10^3$ & $10^4$ & $10^5$ & $10^6$ \\
$m$  & 231 & 1,282 & 10,608 & 101,310 & 1,002,823 \\
$N$  & 261 & 1,322 & 10,708 & 101,510 & 1,003,223 \\
\hline
$J_m(c)$  & 0.25898E-59 & 0.45624E-62 & 0.42252E-63 & 0.13902E-63 & 0.57054E-64 \\
$J_c(c)$  & 0.96366E-01 & 0.44730E-01 & 0.20762E-01 & 0.96369E-02 & 0.44730E-02 \\
\hline
$|1-X_m/J_m(c)|$  & 0.72561E-14 & 0.28169E-12 & 0.13717E-12 & 0.72506E-12 & 0.25122E-10 \\
$|1-Y_m/J_m(c)|$  & 0.64024E-15 & 0.28275E-13 & 0.12375E-12 & 0.13185E-11 & 0.38545E-11 \\
\end{tabular}
\end{center}
\caption{\it
Experiment 3. Parameters: $\delta=150$.
}
\label{t:test515_delta150}
\end{table}

Also, in Figures~\ref{fig:519a_x}, \ref{fig:519a_y} we plot the relative
errors of $X_1,Y_1$, respectively, on a logarithmic scale
as functions of $\log_{10}(c)$. More specifically, each of
Figures~\ref{fig:519a_x}, \ref{fig:519a_y} contains
five plots of such errors, corresponding to $a=1,2,3,4,6$, respectively.
Each point on such plot is the geometric mean of ten relative
errors (corresponding to ten different values of $\delta$
between 50 and 200). For example, to generate
plots corresponding to $a=2$ in Figure~\ref{fig:519a_x}, we use the data from 
Tables~\ref{t:test513a2_delta50}--\ref{t:test513a2_delta150}
(as well as the data corresponding to seven other values of $\delta$).

To each plot in Figures~\ref{fig:519a_x}, \ref{fig:519a_y},
one can fit a line (in the least square sense). The slopes of such lines
are displayed in Table~\ref{t:test519}. This table has the following
structure. Each column corresponds to a different value of $a$.
Second row contains the slopes corresponding to $\text{\rm{rel}}(Y_1)$
(see Figure~\ref{fig:519a_y}).
Third row contains the slopes corresponding to $\text{\rm{rel}}(X_1)$
(see Figure~\ref{fig:519a_x}).
Fourth row contains $\beta(a)$, where $\beta(a)$ is defined
via
\eqref{eq_exp3_11} below (the values in third and fourth rows
would be identical if $\text{\rm{rel}}(X_1)$ were proportional
to $c^{\beta(a)}$).
Last row contains the number $4 \cdot a/(a+2)$
(the power of $c$ in \eqref{eq_ac_6} of Theorem~\ref{thm_ac}).

{\bf Observations.} Several observations can be made
from Tables~\ref{t:test513a2_delta50}--\ref{t:test513a4_delta150},
Figure~\ref{fig:519a}, Table~\ref{t:test519},
and some additional numerical experiments by the author.

{\bf Observation 1.} 
For every choice of parameters in Experiment 1, the coordinate $X_1$
is fairly small compared to $X_k$, as predicted by Theorem~\ref{thm_vec_grow}
and Corollary~\ref{cor_vec_grow}
in Section~\ref{sec_vectors}
(for all $c$,
$X_1/X_k \approx 10^{-25}, 10^{-50}, 10^{-75}$ for
$\delta = 50,100,150$, respectively).
Despite this fact, both $X_1$ and $Y_1$
are still evaluated to fairly high {\it relative} accuracy,
in all cases.

{\bf Observation 3.} For any $c$ and $a$, 
the relative accuracy of both $X_1$ and $Y_1$ seems to
be essentially independent of their magnitude. For example, for $a=4$
and $c=10^6$, the relative accuracy of $X_1$ is 0.4E-8, 0.6E-8, 0.4E-8
for $\delta=50,100,150$, respectively (despite the fact that $X_1$ itself
is equal to 0.4E-33, 0.6E-58, 0.1E-85, respectively).
In other words, the $\delta$-dependent factor in 
\eqref{eq_ac_6} of Theorem~\ref{thm_ac} seems to be an 
artifact of the analysis.

{\bf Observation 4.} On the other hand, the relative accuracy of both $X_1$ and
$Y_1$ does depend on $c$ (as Theorem~\ref{thm_ac} suggests).
In particular, for any fixed $a$, the relative error of $Y_1$ seems to 
be roughly proportional to $c$, e.g.
\begin{align}
\text{\rm{rel}}(Y_1) = O(c) \cdot \varepsilon,
\label{eq_exp1_10}
\end{align}
where $\varepsilon$ is the machine precision (see second row
in Table~\ref{t:test519}).
 
{\bf Observation 5.} For any fixed $a$, the relative error of $X_1$ seems to 
be roughly proportional to $c^\beta$, where $\beta=\beta(a)$ 
is defined via the formula
\begin{align}
\beta(a) = \frac{2 \cdot a}{a+2}
\label{eq_exp3_11}
\end{align} 
(see third and fourth rows in Table~\ref{t:test519},
and also Conjecture~\ref{conj_exp1}).
On the other hand,
in Theorem~\ref{thm_ac} in Section~\ref{sec_asym} 
we derived a certain upper bound 
on the relative error of $X_1$ (see \eqref{eq_ac_6}
and last row in Table~\ref{t:test519});
this bound is proportional to $c^{4a/(a+2)}$. 
In other words, numerical experiments seem to indicate that
Theorem~\ref{thm_ac} overestimates the number of lost digits
roughly by a factor of two. For example, for $a=4$, $\delta=150$ and $c=10^6$
(see last column in 
Table~\ref{t:test513a4_delta150}) we lose almost $\beta(a) \cdot 6 = 8$
decimal digits, while the pessimistic estimate from Theorem~\ref{thm_ac}
suggest that we will lose 16 decimal digits. In other words,
the estimate from Theorem~\ref{thm_ac} is overly cautious.

\begin{figure}[ht]
\centering
\subfigure[$\log_{10}(\text{\rm{rel}}(X_1))$ as a function of $\log_{10}(c)$.]{
\includegraphics[width=11.5cm, bb=81 227 529 564, clip=true]
{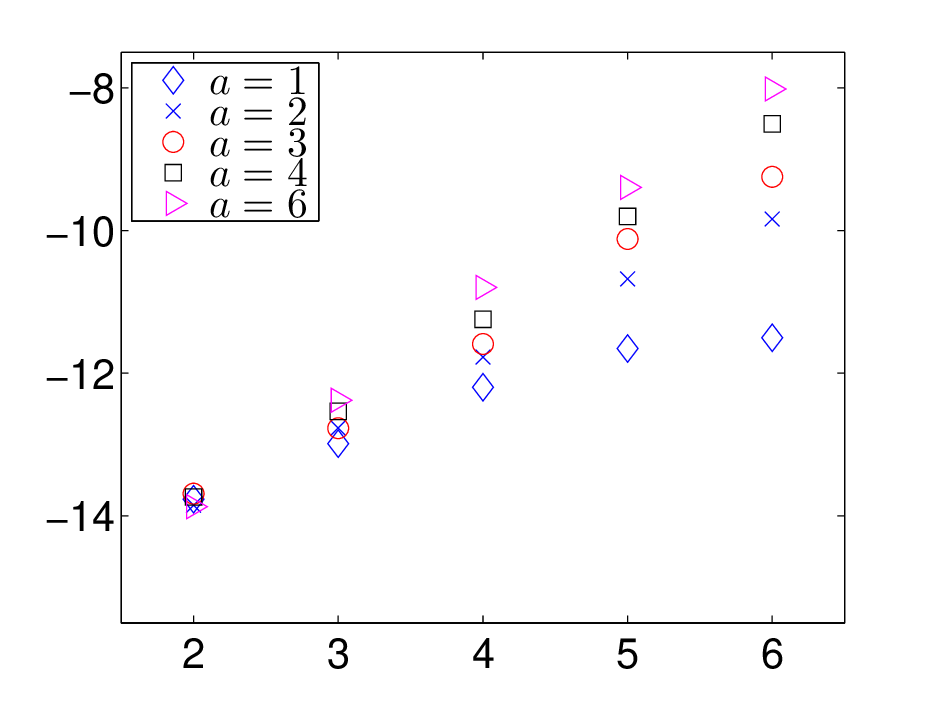}
\label{fig:519a_x}
}
\subfigure[$\log_{10}(\text{\rm{rel}}(Y_1))$ as a function of $\log_{10}(c)$.]{
\includegraphics[width=11.5cm, bb=81 227 529 564, clip=true]
{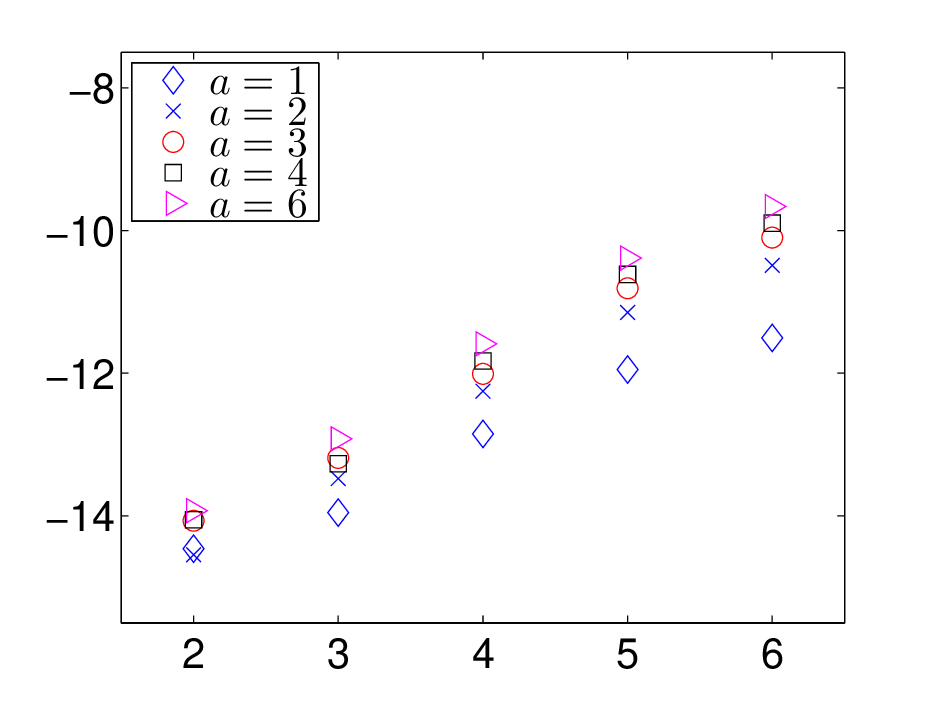}
\label{fig:519a_y}
}
\caption{ 
Relative errors of $X_1,Y_1$, on a logarithmic scale, as a function
of $\log_{10}(c)$, for $a=1,2,3,4,6$. Corresponds to Experiment 1.
}
\label{fig:519a}
\end{figure}

\subsection{Experiment 2.}
\label{sec_mat_exp2}

In Experiment 1, we took a rather detailed look at relative errors
to which the first coordinate of an eigenvector of certain
tridiagonal matrices is evaluated.
The purpose of this section is to illustrate the analysis
of Section~\ref{sec_analytical} in a more qualitative way.

To that end, we carry out the experiment described in 
Section~\ref{sec_mat_exp1} with the following parameters:
$a=2$, $c=1000$, $n=1497$, $\delta=50$
(see Table~\ref{t:test513a2_delta50}).
We obtain the four unit-length vectors $X,Y,\hat{X},\hat{Y}$ in $\Rc^n$,
as described in Section~\ref{sec_mat_exp1}.

{\bf Figures.}
We display the results of this experiment in 
Figures~\ref{fig:513_x}--\ref{fig:513_erry}.
In each figure, the abscissa corresponds to the indices of
the eigenvector, i.e. $1 \leq j \leq n$;
thus, we plot certain functions of the indices of the eigenvector.

In Figure~\ref{fig:513_x},
we plot the coordinates $X_j$ of $X$, on the 
linear scale (left) and on the logarithmic scale (right).

In Figure~\ref{fig:513_errx}, 
we plot the relative (left) and absolute (right) errors of $X_j$
on the logarithmic scale.

In Figure~\ref{fig:513_erry}, 
we plot the relative (left) and absolute (right) errors of $Y_j$
on the logarithmic scale.

{\bf Observations.} 
Several observations can be made
from Figures~\ref{fig:513_x}--\ref{fig:513_erry}.

The following three observations pertain to the
behavior of the coordinates of $X$ (see Figure~\ref{fig:513_x}).

{\bf Observation 1.} In the beginning,
the coordinates of $X$ 
grow rapidly from $\approx 10^{-26}$ to $\approx 10^{-1}$ up
to the index $k$ such that $\lambda \approx A_k+2$
(in agreement with Theorem~\ref{thm_vec_grow}
in Section~\ref{sec_vectors}). 
We refer to the corresponding indices as the "region of growth".

{\bf Observation 2.}
At the other end,
they decay rapidly (while changing signs) 
from $\approx 0.05$ to $\approx 10^{-14}$, starting from
the index $n-q$ such that $\lambda \approx A_{n-q}-2$
(in agreement with Theorem~\ref{thm_vec_decay}
in Section~\ref{sec_vectors}).
We refer to the corresponding indices as the "region of decay".

{\bf Observation 3.} In the middle (i.e. for indices $j$
such that $\lambda-2 \leq A_j \leq \lambda+2$), the coordinates
behave in an "oscillatory way" (see e.g. Figure~\ref{fig:513_x}).
Such behavior is expected 
from Theorems~\ref{thm_mat}, \ref{thm_delta} and 
Corollaries~\ref{cor_mat}, \ref{cor_delta}
in Section~\ref{sec_vectors}.
We refer to the corresponding indices as the "oscillatory region"
(see also \cite{Report4} for an alternative approach to the
evaluation of $X_j$ in the oscillatory region that,
{\it inter alia}, further
justifies this term).

The following observations pertain to the behavior
of relative and absolute errors to which the coordinates
of the eigenvector are evaluated, by either Inverse Power
or the algorithm from Section~\ref{sec_short}.

{\bf Observation 4.} Qualitatively, the behavior of relative
errors of $X_j$ is similar to that of $Y_j$ and depends
of whether $j$ is in the region of growth, in the region of decay,
or in the oscillatory region.

{\bf Observation 5.} In the region of growth, the relative errors
of $X_j$ change monotonically with $j$ and always stays
"small" (below $10^{-12}$), in agreement with
Theorems~\ref{thm_err_grow_rel}, \ref{thm_summary_left},
Corollary~\ref{cor_glue} in Section~\ref{sec_about_errors}
and Theorem~\ref{thm_ac} in Section~\ref{sec_asym}. 
In the region of decay, the relative
errors of $X_j$ display a similar behavior,
in agreement with Theorem~\ref{thm_summary_right},
Corrolary~\ref{cor_glue}
in Section~\ref{sec_about_errors},
and Theorem~\ref{thm_ac} in Section~\ref{sec_asym}.
In particular, both in the regions of growth and 
in the region of decay
the relative errors of $X_j$ essentially do not depend on 
the magnitude of $X_j$.

{\bf Observation 6.} In the oscillatory region, the relative errors
of $X_j$ oscillate between $10^{-16}$ and $10^{-10}$.
On the other hand, the {\it absolute} errors of $X_j$ always stay below
roughly $10^{-14}$. In other words, the relative errors of $X_j$
in the oscillatory region depend on the magnitude of $X_j$,
in agreement with Theorems~\ref{thm_osc}, \ref{thm_summary_left}
in Section~\ref{sec_about_errors}.

\subsection{Experiment 3.}
\label{sec_mat_exp3}

In this experiment, we illustrate the numerical algorithms
of Section~\ref{sec_num_algo} via evaluation of Bessel functions
(see Sections~\ref{sec_prel_bessel},
\ref{sec_num_bessel}, \ref{sec_mat_bessel}).

{\bf Description.}
We first choose, more or less arbitrarily,
the real number $\delta \geq 0$. Then, for each choice
of five different values $c=10^2,10^3,10^4,10^5,10^6$,
we do the following. We define the integer 
$m=m(\delta,c)$ via the formula
\begin{align}
m = c + \delta^{2/3} \cdot c^{1/3}
\label{eq_exp3_01}
\end{align}
(see \eqref{eq_aa_02} in Theorem~\ref{thm_aa} and
\eqref{eq_aa_33} in Theorem~\ref{thm_aaa}), 
select the integer $N>m$ (according to Remark~\ref{rem_bessel_n} 
in Section~\ref{sec_num_bessel}),
define the integer $n$ via the formula
\begin{align}
n = 2 \cdot N + 1,
\label{eq_exp3_02}
\end{align}
define $A_1,\dots,A_n$ via \eqref{eq_ac_2} with $a=1$ in Theorem~\ref{thm_ac},
and then define
the symmetric tridiagonal $n \times n$ matrix $A=A(c)$ via
\eqref{eq_running}. Then, we define the real number $\lambda(c)$
via the formula
\begin{align}
\lambda(c) = 2 + \frac{n+1}{c}.
\label{eq_exp3_03}
\end{align}
(We observe that $\lambda(c)$ is an eigenvalue of $A$,
according to \eqref{eq_bessel_num_lambda} in Section~\ref{sec_num_bessel}.)

Next, we obtain the unit-length $\lambda(c)$-eigenvector of $A$ by four
different methods:

{\bf 1.} $Y=(Y_N,\dots,Y_0, \dots, Y_{-N})$ via 30 iterations of
Shifted Inverse Power, in double precision (observe that the indices vary
between $N$ and $-N$, as in \eqref{eq_bessel_num_vec}).

{\bf 2.} $X=(X_N,\dots,X_0, \dots X_{-N})$ 
via the algorithm from Section~\ref{sec_short},
in double precision.

{\bf 3.} $\hat{Y}=(\hat{Y}_N,\dots,\hat{Y}_0,\dots,\hat{Y}_{-N})$ 
via 30 iterations of Shifted Inverse Power, in extended precision.

{\bf 4.} $\hat{X}=(\hat{X}_N,\dots,\hat{X}_0,\dots,\hat{X}_{-N})$ 
via the algorithm from Section~\ref{sec_short},
in extended precision.

The experiment is conducted for each pair of values $\delta,c$, where
$\delta = 50,100,150$ and
$c=10^2,10^3,10^4,10^5,10^6$.
In each case,
we verify that each of $\hat{X}$ and $\hat{Y}$ satisfies the definition of an
eigenvector coordinate-wise 
to at least 17 decimal digits, 
and also that $\hat{X}=\hat{Y}$ to at least 17 decimal digits.
In other words, each of $\hat{X}, \hat{Y}$ is the unit-length
$\lambda(c)-$eigenvector of $A$ defined to full double precision.
Also, we verify that
the middle $2 \cdot m+1$ coordinates of both $\hat{X}$ and $\hat{Y}$
are equal to $J_m(c),\dots,J_0(c),\dots,J_{-m}(c)$ to at least
17 decimal digits (see Remark~\ref{rem_bessel_n}
in Section~\ref{sec_num_bessel}).
We use these observations to compute the accuracy to which
the coordinates $X_m,\dots,X_0$ of $X$ and $Y_m,\dots,Y_0$ of $Y$ approximate
$J_m(c),\dots,J_0(c)$.

The results of the experiment are displayed in 
Tables~\ref{t:test515_delta50}--\ref{t:test515_delta150}.
Each of these tables corresponds to a particular choice
$\delta$ in \eqref{eq_exp3_01}, and has the following structure.
Each of five columns corresponds to a different value of $c$,
between $10^2$ and $10^6$. The first three rows contain
$c$, the integer $m$ defined via \eqref{eq_exp3_01},
and the integer $N>m$ (see Remark~\ref{rem_bessel_n}
in Section~\ref{sec_num_bessel}). 
The next two rows contain
$J_m(c)$ and $J_c(c)$.
The last two rows contain the relative accuracy to which $X_m$ and $Y_m$,
respectively, approximate $J_m(c)$.

{\bf Observations.} Several observations can be made from
Tables~\ref{t:test515_delta50}--\ref{t:test515_delta150}.

{\bf Observation 1.} 
For every choice of parameters in Experiment 3, 
$J_m(c)$ is fairly small compared to $J_c(c)$,
as predicted by Theorem~\ref{thm_vec_grow}
and Corollary~\ref{cor_vec_grow} in Section~\ref{sec_vectors}
(for all $c$,
$J_c(c) \leq 10^{-20}, 10^{-40}, 10^{-59}$ for
$\delta = 50,100,150$, respectively).
Despite this fact,
both $X_m$ and $Y_m$
approximate $J_m(c)$ to 
a fairly high {\it relative} accuracy,
in all cases. Moreover, for any $c$, this accuracy
seems to be independent of the magnitude of $J_m(c)$
(compare to \eqref{eq_ac_6} of Theorem~\ref{thm_ac};
 see also Conjecture~\ref{conj_exp1}).

{\bf Observation 4.} On the other hand, the relative accuracy of both $X_1$ and
$Y_1$ does depend on $c$ (as Theorem~\ref{thm_ac} 
in Section~\ref{sec_asym} suggests).
In particular, for any fixed $a$, the relative error of $Y_1$ seems to 
be roughly proportional to $c^{0.8}$, e.g.
\begin{align}
\text{\rm{rel}}(Y_1) = O(c^{0.8}) \cdot \varepsilon,
\label{eq_exp3_10}
\end{align}
where $\varepsilon$ is the machine precision (see second column
in Table~\ref{t:test519}). Also, the relative error of $X_1$
seems to be roughly proportional to $c^{2/3}$ 
(see Table~\ref{t:test519}),
in agreement with 
Conjecture~\ref{conj_exp1} above
(compare to \eqref{eq_ac_6} of Theorem~\ref{thm_ac}).

%
%

\clearpage
\begin{figure}[ht]
\centering
\subfigure[coordinates: linear and logarithmic scales]{
\includegraphics[width=6cm, bb=81 217 529 564, clip=true]
{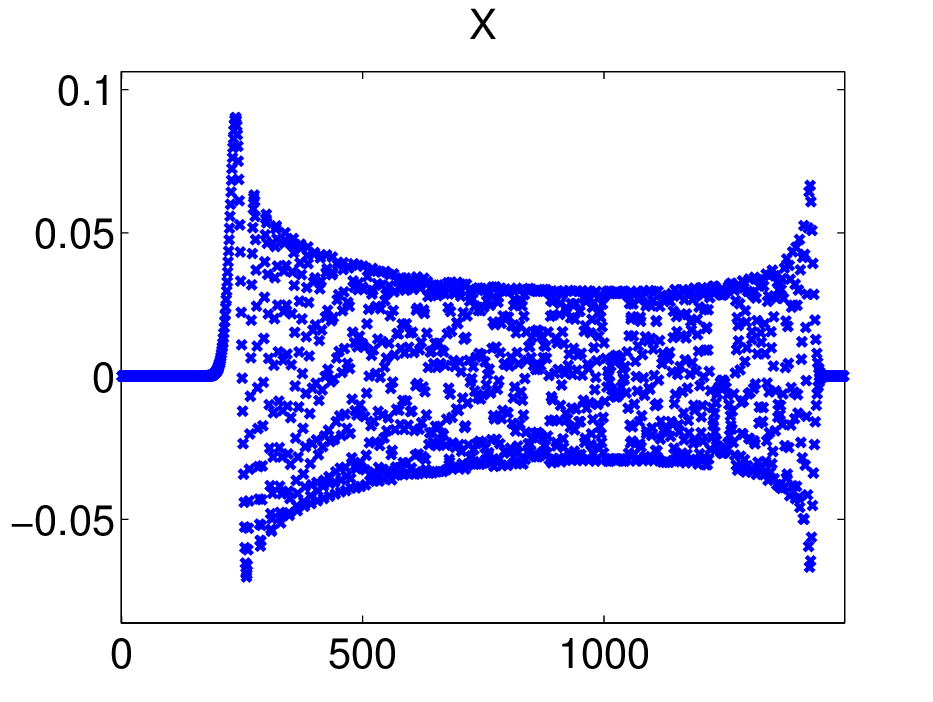}
\includegraphics[width=6cm, bb=81 217 529 564, clip=true]
{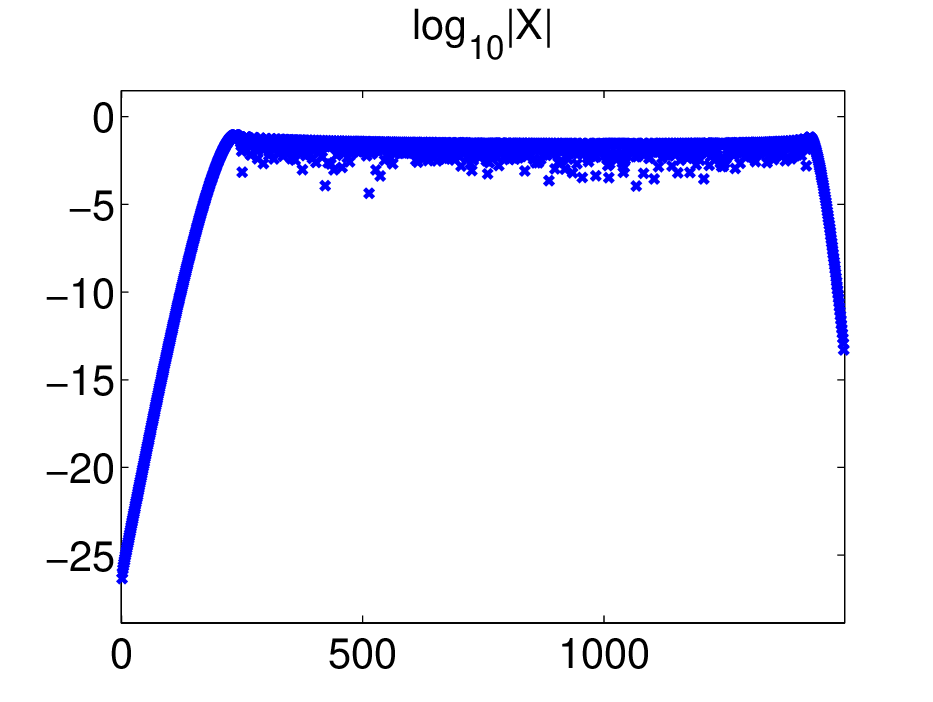}
\label{fig:513_x}
}
\subfigure[principal algorithm: relative and absolute errors]{
\includegraphics[width=6cm, bb=81 217 529 564, clip=true]
{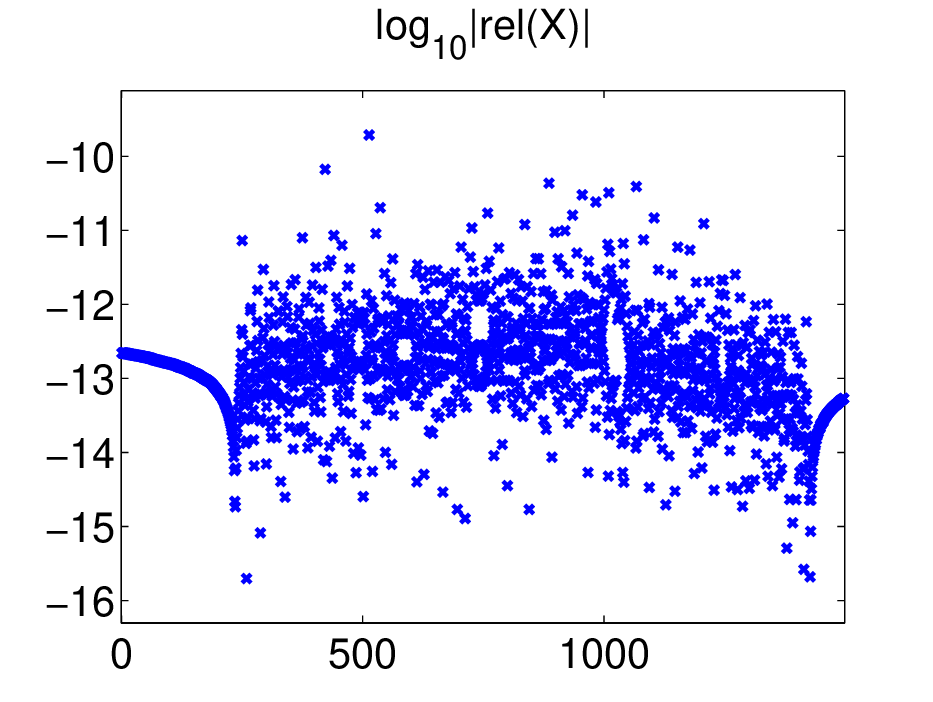}
\includegraphics[width=6cm, bb=81 217 529 564, clip=true]
{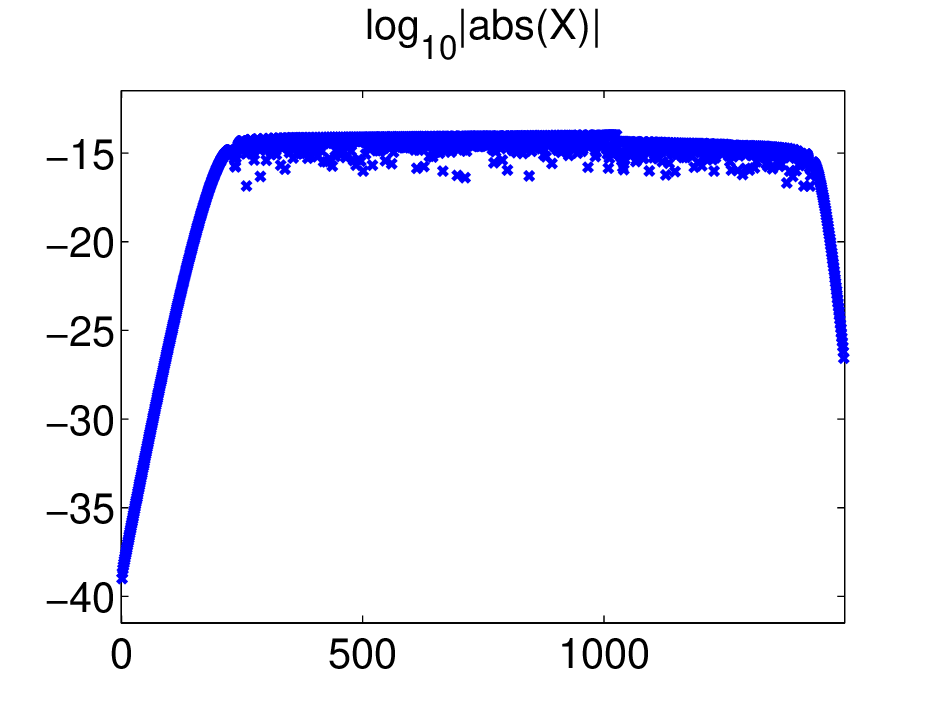}
\label{fig:513_errx}
}
\subfigure[inverse power: relative and absolute errors]{
\includegraphics[width=6cm, bb=81 217 529 564, clip=true]
{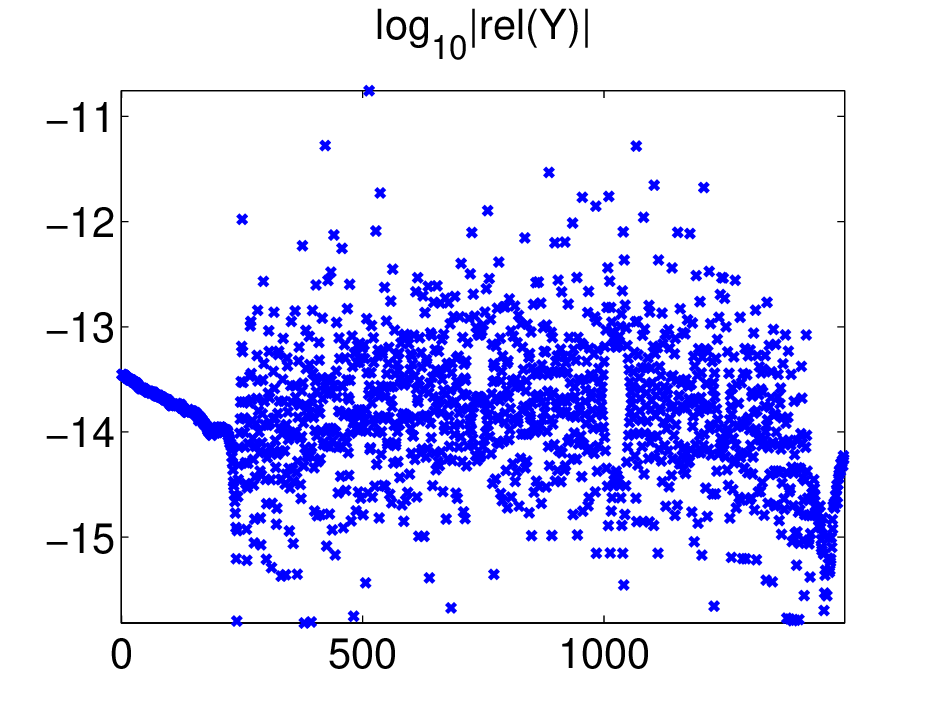}
\includegraphics[width=6cm, bb=81 217 529 564, clip=true]
{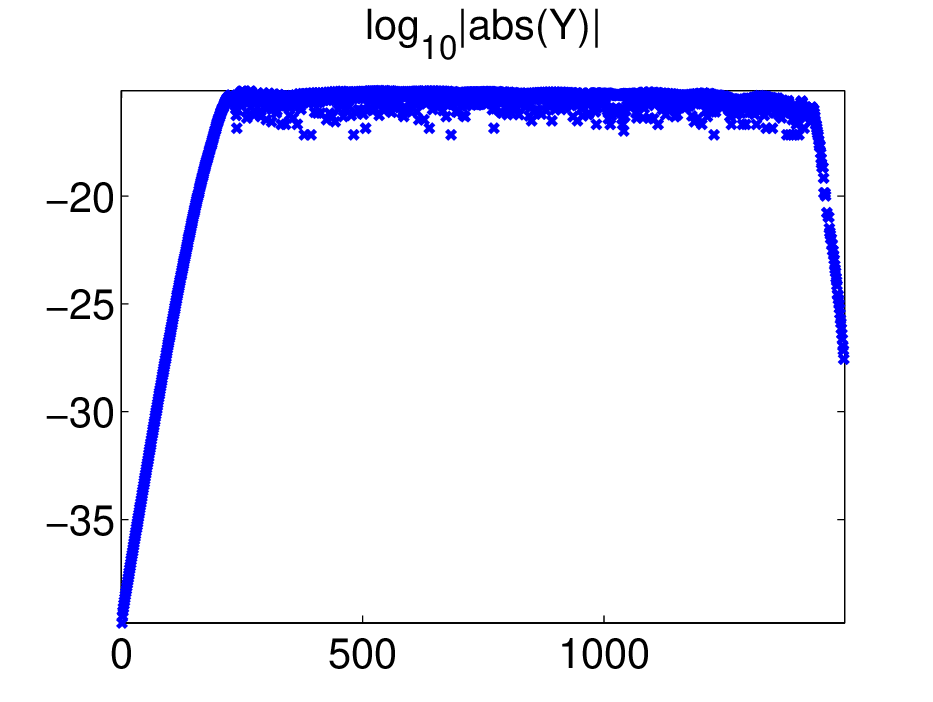}
\label{fig:513_erry}
}
\caption{ 
The coordinates of $X$ (principal algorithm)
and $Y$ (30 iterations
of Inverse Power).
Parameters: $c=1000$,
$n=1500$, 
$\lambda=\mbox{\text{\rm{0.41022E+01}}}$, 
$k=226$, $q=65$. Corresponds to Experiment 2.
}
\label{fig:513}
\end{figure}


\end{document}